\documentclass[11pt,a4paper,reqno]{amsart}

\usepackage[english]{babel}
\usepackage{enumitem}
\usepackage{bbm}
\usepackage{color}
\usepackage{amsmath,amsfonts,amssymb,amsopn,amscd,amsthm}

    \newtheorem{definition}{Definition}[section]
    \newtheorem{lemma}[definition]{Lemma}
    \newtheorem{theorem}[definition]{Theorem}

    \theoremstyle{remark}
    
    \newtheorem{remark}[definition]{Remark}

\title{Precise Limit Theorems for Lacunary Series}
\author[F.~Delbaen]{ Freddy Delbaen} 
\address[F.~Delbaen]{Department of Mathematics, ETH Z\"urich, R\"amistrasse 101, 8092, Z\"urich, Switzerland \,and  \,Institut f\"ur Mathematik, Universit\"at Z\"urich, Winterthurerstrasse 190, 8057 Z\"urich, Switzerland}
\email{delbaen@math.ethz.ch, }
\author[E.~Hovhannisyan]{Emma Hovhannisyan} 
\address[E.~Hovhannisyan]{\,Institut f\"ur Mathematik, Universit\"at Z\"urich, Winterthurerstrasse 190, 8057 Z\"urich, Switzerland}
\email{emma.hovhannisyan@math.uzh.ch}

\renewcommand{\d}{\delta}
\newcommand{\half}{\frac{1}{2}}

\newcommand{\g}{\gamma}

\newcommand{\N}{\mathbb{N}}
\newcommand{\R}{\mathbb{R}}
\newcommand{\C}{\mathbb{C}}
\renewcommand{\P}{\mathbb{P}}
\newcommand{\E}{\mathbb{E}}

\newcommand{\ckn}{c_{k, n}}
\newcommand{\akn}{a_{k, n}}

\renewcommand{\a}{\alpha}

\renewcommand{\l}{\lambda}

\begin{document}
 \maketitle \vspace{-5mm}

\begin{abstract}
Lacunary trigonometric and Walsh series satisfy limiting results that are typical for i.i.d. random variables such as the central limit theorem \cite{SZ47}, the law of the iterated logarithm \cite{W59} and several probability related limit theorems. For H{\"o}lder continuous, periodic functions this phenomenon does not hold in general. In \cite{K46} and \cite{K49}, the validity of the central limit theorem has been shown for the sequence $\left(f(2^k x)\right)_k$ and in the case of ``big gaps''. In this paper, we present an alternative approach to prove the above theorem based on martingale theory, which allows us to generalize the theorem to infinite product spaces of arbitrary probability spaces, equipped with the shift operator. 

In addition, we show the local limit theorems for lacunary trigonometric and Walsh series, and for H{\"o}lder continuous, periodic functions in the case of ``big gaps''. We also establish Berry-Esseen bounds and moderate deviations for lacunary Walsh series. Furthermore, we identify the scale at which the validity of the Gaussian approximation for the tails breaks. To derive these limiting results, the framework of mod-Gaussian convergence has been used.
\end{abstract}

\section{Introduction}
\subsection{Limit theorems for lacunary trigonometric and Walsh series and H\"older continuous periodic functions}
If $m_1 < m_2 < \dots $ is an infinite sequence of integers such that for some $q>1$ and for all $k\in \N$
\begin{align}\label{lacunary}
\frac{m_{k+1}}{m_k} \geq q >1,
\end{align} 
we say that it is lacunary. 

For such a sequence let  us consider the following trigonometric series:
\begin{align}\label{tryg_s}
\sum_{k=1}^{\infty} \big( a_k \cos(2 \pi m_k x) + b_k  \sin(2 \pi m_k x)\big), \quad 0 \leq x \leq 1.
\end{align}
It is well known that such a trigonometric series (\ref{tryg_s}) satisfies  limiting results that are typical for independent and idenitically distributed random variables. For instance, Salem and Zygmund \cite{SZ47} proved the following theorem.
\begin{theorem}[CLT for trigonometric series \cite{SZ47}] \label{SZ}
Let $S_n(x) = \sum_{k=1}^n \big( a_k \cos(2 \pi m_k x) + b_k  \sin(2 \pi m_k x)\big)$, $0 \leq x \leq 1$ be the $n$th partial sum of (\ref{tryg_s}) for a lacunary sequence $\left( m_k \right)_{k \in \N}.$ If 
\begin{align*}
A_n := &\left( \half \sum_{k=1}^n \left( a_k^2 + b_k^2 \right) \right)^{\frac{1}{2}} \to \infty,\\
&\sqrt{a_n^2 + b_n^2} = o(A_n),
\end{align*}
then for $n \to \infty,$
\begin{align}\label{tryg_CLT}
\left| \left\lbrace  x \in [0,1]; \frac{S_n(x)}{ A_n} \leq t \right\rbrace  \right| \to \frac{1}{\sqrt{2 \pi}} \int_{-\infty}^t e^{-\frac{u^2}{2}} du.
\end{align}
Here $| \cdot |$ denotes Lebesgue measure.
\end{theorem}
In a subsequent paper Salem and Zygmund \cite{SZ48} extended the central limit theorem to the case where the $m_k$'s are not necessarily integers.
Erd\H{o}s \cite{E62} relaxed the lacunarity assumption (\ref{lacunary}) and showed that the central limit theorem (\ref{tryg_CLT}) still holds for the trigonometric series $\sum_{k=1}^n \cos(2 \pi m_k x)$ under the assumption 
\begin{align}
\frac{m_{k+1}}{m_k} > 1+\frac{c_k}{\sqrt{k}},\label{Erdos_lac}
\end{align}
with $c_k \to \infty.$ Moreover, the condition (\ref{Erdos_lac}) is optimal in the sense that for every $c$ there is a sequence $\left(m_k \right)_{k \in \N}$ satisfying $\frac{m_{k+1}}{m_k} > 1+\frac{c}{\sqrt{k}}$ such that the CLT does not hold. 
Under the lacunarity condition (\ref{lacunary}), \cite{W59}  
proved that if 
\begin{align*}
a_n = o\left(\frac{A_n}{\sqrt{\log \log (A_n) }}\right),
\end{align*}
then the trigonometric series $\sum_{k=1}^n a_k \cos(2 \pi m_k x) $ obeys the law of the iterated logarithm (LIL).

Using martingale techniques, Philipp and Stout \cite{PS75} established the almost sure invariance principle for the trigonometric series $\sum_{k=1}^n a_k \cos(2 \pi m_k x).$ More precisely, they showed that if 
\begin{align*}
a_n = o\left(A_n^{1-\delta} \right)
\end{align*}
 for some $\delta > 0,$ then on a richer probability space for $\l < \frac{\delta}{32}$ there exists a Brownian motion  $\lbrace X_t, t \geq 0 \rbrace$ such that for  $A_n^2 \leq t < A_{n+1}^2,$ 
\begin{align*}
\sum_{k=1}^n a_k \cos(2 \pi m_k x)  - X_t <<  t^{\frac{1}{2} - \l} \quad \mathrm{a.s.}.
\end{align*}
 This matching of trajectories of the trigonometric series with the trajectories of the Brownian motion allows to deduce directly the wide range of limiting results such as the CLT, LIL, Chung's LIL and the arcsin law. 
 
Morgentaler \cite{M57} investigated statistical properties of Walsh series and proved a central limit theorem for subsequences of lacunary Walsh series under similar assumptions as for lacunary trigonometric series. To define Walsh functions, we need the Rademacher functions $\left(r_n(x) \right)_{n \geq 0} $, introduced by Rademacher \cite{R22}, and constructed as follows:
\begin{align*}
r_n(x) = r_0(2^nx),
\end{align*}
where 
\begin{align*}
&r_0(x) = \begin{cases} 1, & \mbox{if } 0 \leq x < \frac{1}{2} \\ -1, & \mbox{if } \frac{1}{2} \leq x <1 \end{cases}.
\end{align*}
The Rademacher functions $\left(r_n(x) \right)_{n \geq 0} $ form an orthonormal system on $[0,1].$

Let $n\in \mathbb{N}$ have following unique dyadic expansion $n = \sum_{i=0}^\infty k_i 2^i,$ where $k_i \in \lbrace 0,1 \rbrace$ and $l_1, l_2, \dots, l_m$ be the coefficients for which $k_{l_i} = 1.$ 
The $n$th Walsh function $W_n(x)$ is defined as 
\begin{align}\label{Walsh}
W_n(x) = \prod_{i=1}^m r_{l_i}(x) = \prod_{i=1}^\infty r_i^{k_i} (x).
\end{align} 
\begin{theorem}[CLT for Walsh series \cite{M57}] \label{Walsh_CLT}
Let $S_n(x) = \sum_{k=1}^n a_k W_{m_k}(x)$, $0 \leq x \leq 1$ be the $n$th partial sum of the Walsh series (\ref{Walsh}) with the sequence $(m_k)_{k \in \N}$ satisfying the condition (\ref{lacunary}). If
\begin{align*}
A_n= &\left( \sum_{k=1}^n a_k^2 \right)^{\frac{1}{2}} \to  \infty,\\
&a_n = o(A_n),
\end{align*}
then for $n \to \infty$
\begin{align}\label{Walsh_CLT}
\left| \left\lbrace  x \in [0,1]; \frac{S_n(x)}{ A_n} \leq t \right\rbrace  \right| \to \frac{1}{\sqrt{2 \pi}} \int_{-\infty}^t e^{-\frac{u^2}{2}} du.
\end{align}
 
\end{theorem}
F{\"o}ldes (\cite{F72, F75}) and Takahashi (\cite{T75}) deduced the central limit theorem as well as the law of iterated logarithm for the Walsh series under the weaker lacunarity assumptions (\ref{Erdos_lac}).

Next we look at general functions $f$ that are neither cosine nor Walsh functions. In this general case, the arithmetic structure of  the sequence $( m_k )_{k }$  matters and the lacunarity assumption is not sufficient to deduce central limit theorems. Kac \cite{K46} and Fortet \cite{F40} showed that if $f$  satisfies a Lipschitz condition or is  of bounded variation then the central limit theorem holds for $m_k = 2^k.$ 
\begin{theorem}[CLT for the type $\sum f\left( 2^k x\right)$ \cite{K46}]\label{KacT}
Let $f(x)$ be a measurable function defined on $[0, 1)$ and extended  periodically by setting $f(x + 1)= f(x),$ let also 
\begin{align*}
&\int_0^1 f(x) dx = 0,\\
& \lim_{n \to \infty} \frac{1}{n} \int_0^1 \left(\sum_{k=1}^n f\left( 2^k x \right) \right)^2 dx := \sigma^2 \neq 0. 
\end{align*}
Let $f(x) = \sum_{n=1}^\infty e_n \cos(2 \pi n x) $ be its Fourier expansion such that
\begin{align*}
|e_n| < \frac{M}{n^{\beta}}, \quad \beta > \half, \quad  n \in \N,
\end{align*}
or let $f$ satisfy a H{\"o}lder continuity condition,
then
\begin{align*}
\left| \left\lbrace  x \in [0,1]; \frac{1}{\sqrt{n}}\sum_{k=1}^n f(2^k x) \leq t\right\rbrace  \right| \to \frac{1}{\sigma \sqrt{2 \pi} } \int_{-\infty}^t e^{-\frac{u^2}{2\sigma^2 }} du.
\end{align*}
\end{theorem}
Note that when instead of the lacunarity assumption, one considers the case of ``big gaps'' i.e. 
\begin{align}
\lim_{k \to \infty} \frac{m_{k+1}}{m_k} \to \infty, \label{biggap}
\end{align}
the arithmetic structure of the sequence $( m_k)_{k \geq 1}$ becomes irrelevant and the central limit theorem holds  (\cite{K49, T61}).
\begin{theorem}[CLT for ``big gaps'' \cite{K49}]
Let $f(x)$ be a measurable function satisfying the assumptions of Theorem \ref{KacT} and $( a_k)_{k \geq 1} $ be a sequence of real numbers such that  
\begin{align*}
A^2_n := \sum_{k=1}^n a_k^2 \to \infty, \quad \max_{1 \leq k \leq n} |a_k| = o(A_n),
\end{align*}
then 
\begin{align*}
\left| \left\lbrace  x \in [0,1]; \frac{1}{A_n} \sum_{k=1}^n a_k f(m_k x) \leq t \right\rbrace  \right| \to \frac{1}{\sigma \sqrt{2 \pi} } \int_{-\infty}^t e^{-\frac{u^2}{2 \sigma^2 }} du,
\end{align*}
where $( m_k )_k$ is a sequence of integers satisfying the condition (\ref{biggap}).
\end{theorem}
Gaposhkin \cite{G70}, and Aistleitner and Berkes \cite{AB10} studied necessary and sufficient conditions for $\left( f (m_k x) \right)_k$ to satisfy the central limit theorem. The conditions are connected with a number of solutions of a certain Diophantine equation.  The further results on asymptotic properties of the periodic H{\"o}lder continuous functions as well as for the periodic functions of bounded variation can be found in \cite{A10, AE12, B78, F94, F08, G66, I51, M50, P92}.

In this article, we have succeeded in proving local limit theorems (LLT) and other fine asymptotic results for lacunary trigonometric and Walsh series, and H\"older continuous, periodic functions. To derive these asymptotic results we have used the machinery of mod-Gaussian convergence that is presented in the next subsection. 

 The structure of the paper goes as follows. The main results of the paper are presented in Section \ref{main_results}. In particular, there we claim the validity of the mod-Gaussian convergence for  lacunary trigonometric and Walsh series, and H\"older continuous, periodic functions under some additional assumptions. Note that the local limit theorem will be one of the consequences of the mod-Gaussian convergence.  In addition, using martingale techniques we present an alternative approach to prove a central limit theorem for the sequence $\left(f(2^k x)\right)_k$ (see  Kac \cite{K46}). The latter allows us to generalize this theorem to infinite product spaces of arbitrary probability spaces. The proofs of these results can be found in Sections \ref{Section:3}, \ref{Section:4} and  \ref{Section:5}. 
 
\subsection{The notion of mod-Gaussian convergence}
Recently a new probabilistic tool, mod-$\phi$ convergence, was introduced and developed in the articles \cite{DKN15, JKN11, KN10, KN12, BKN14, FMN16, FMN17, BMN17} sometimes with small variations in the definition, in connection with several problems and examples coming from various areas of mathematics such as number theory, graph theory, random matrix theory \cite{BHR17C, BHR17B, BHR17A}, non-commutative and classical probability theory. The main idea of mod-$\phi$ convergence arises from a sequence of random variables that does not converge in distribution, i.e. a sequence of their characteristic functions do not converge pointwise, but nevertheless, after some renormalization it converges to some limiting function. In the scope of this article, we are only interested in a special case of mod-$\phi$ convergence, the so called mod-Gaussian convergence. Therefore, from now one we only discuss the mod-Gaussian convergence. For the details on the      of the mod-$\phi$ convergence, we kindly refer to \emph{loc. cit.}. 

One of the important aspects of the mod-Gaussian convergence is that it implies results such as local limit theorems \cite{DKN15, KN12, BMN17}, speed of convergence in the central limit theorem \cite{FMN17} and moderate deviations \cite{FMN16}.  These asymptotic behaviors, well known for instance for sums of independent and identically distributed random variables, can also be deduced for sequences converging in mod-Gaussian sense, which are neither independent nor identically distributed.

In what follows, we first introduce the general definition of mod-Gaussian convergence. 
For $-\infty \leq c < 0 < d\leq \infty$, set 
\begin{align*}
S_{(c,d)}=\{z\in\C, c < \Re (z) < d\}.
\end{align*}
\begin{definition} \label{mod-Gauss}
 Let $\left(X_n\right)_{n\in\N}$ be a sequence of real-valued random variables, and $\varphi_n(z)=\E\left[e^{zX_n}\right]$ be their moment generating functions, which we assume to exist over the strip $S_{(c,d)}$. 
 We assume that there exists an analytic function $\psi(z)$ not vanishing on the real part of $S_{(c,d)}$, such that locally uniformly on $S_{(c,d)},$
 \begin{align*}
 \lim_{n\to\infty}\varphi_n(z)e^{-t_n\frac{z^2}{2}}=\psi(z),
 \end{align*} 
where $\left(t_n\right)_{n\in\N}$ is some sequence going to infinity. We then say that 
 $\left(X_n\right)_{n\in\N}$ converges mod-Gaussian on $S_{(c,d)},$ with parameters $(t_n)_{n \in \N}$ and limiting function $\psi$.
\end{definition}
This version of mod-Gaussian convergence  implies moderate deviations and extended central limit theorems (see \cite{FMN16}, Theorems 4.2.1 and 4.3.1, respectively). Moreover, under additional assumptions, it is possible to deduce Berry-Esseen estimates (see Theorem 2.16 in \cite{FMN17})  as well as a local limit theorem that we present below. 
\begin{definition}
Let $\left(X_n\right)_{n\in\N}$ be a sequence of real-valued random variables, and $\left(t_n\right)_{n \in \N}$ a sequence growing to infinity. Consider the following assertions:
\begin{enumerate}[label=(\textbf{Z\arabic*}),ref=(Z\arabic*)]
\item \label{Z1} Fix $v, w > 0$ and $\gamma\in \R$.  There exists a zone $[-Dt_n^\g,Dt_n^\g]$, $D>0$, such that, for all
$\l$ in this zone
\begin{align*}
\left|\E\left[e^{i \l X_n} \right] e^{-\frac{t_n \l^2}{2}}-1\right| \leq K_1|\l|^v e^{K_2 |\l|^w} 
\end{align*}
for some positive constants $K_1$ and $K_2$, that are independent of $n$.
\item \label{Z2} $w, \gamma$ and $D$ satisfy 
\begin{equation}
w \geq 2;\qquad -\half \le\gamma\le\frac{1}{w-2};\qquad D\le \left(\frac{1}{4K_2}\right)^{\frac{1}{w-2}}. \nonumber
\end{equation}
\end{enumerate}
If conditions \ref{Z1} and \ref{Z2} are satisfied, we say that we have a zone of control $[-Dt_n^\g,Dt_n^\g]$ with index $(v,w)$. 
\end{definition}

\begin{theorem}[LLT, Theorem 9 in \cite{BMN17}]\ \label{llt}\\
 Let $\left(X_n \right)_{n \in \N}$ be a sequence of real-valued variables for which conditions \ref{Z1} and \ref{Z2} hold. Let $x \in \R$ and $B$ be a fixed Jordan measurable subset with $|B|>0.$ Then for every exponent $\d\in\left(0,\g+\frac{1}{2}\right)$, 
 \begin{equation*}
\lim_{n\to\infty}(t_n)^{\d}\P\left[\frac{X_n}{\sqrt{t_n}}-x\in\frac{1}{t_n^\d} B\right]=\frac{|B|}{\sqrt{2\pi}}.
\end{equation*}
\end{theorem}

However, in some cases, it is impossible to prove the mod-Gaussian convergence in the sense of Definition \ref{mod-Gauss}, but we still want to derive the limiting results from this convergence such as local limit theorems. Here comes to play another version of mod-Gaussian convergence that has been introduced in \cite{DKN15}. 

\begin{definition} \label{mod-Gauss-weak}
 Let $\left(X_n\right)_{n\in\N}$ be a sequence of real-valued random variables with the moment generating functions $\varphi_n( i \l)=\E\left[e^{i \l X_n}\right].$ We assume:
 \begin{enumerate}[label=(\textbf{H\arabic*}),ref=(H\arabic*)]
\item \label{H1} There exists a sequence $\left( A_n \right)_{n \in \N}$ tending to  $\infty$ such that as $n \to \infty$
\begin{align*}
\varphi_n \left( \frac{i \l}{A_n}\right) \to e^{-\frac{\l^2}{2}}.
\end{align*}
\item \label{H2} For all $K \geq 0$, the sequence
$\varphi_n(\frac{i \l}{A_n}) \mathbbm{1}_{|\l| \leq A_n K}$ is uniformly integrable on $\mathbb{R}$.
 \end{enumerate}
 If the properties \ref{H1} and \ref{H2} hold, we say that there is a mod-Gaussian convergence for the sequence $(X_n)_{n \in \N}.$
\end{definition}
In \cite{DKN15} the following theorem was proven.
\begin{theorem}  [LLT for mod-Gaussian convergence in the sense of Definition \ref{mod-Gauss-weak}]\label{llt_del}
Suppose that the mod-Gaussian convergence holds in the sense of Definition \ref{mod-Gauss-weak} for the sequence $(X_n)_{n \in \N}$. Then we have  
\begin{align*}
A_n \E [f(X_n)] \rightarrow \frac{1}{\sqrt{2 \pi}}  \int_\R f(x)dx,
\end{align*}
for all continuous functions $f$  with compact support. More precisely, we have
\begin{align*}
A_n \P[X_n \in B] \to \frac{|B|}{\sqrt{2 \pi}},
\end{align*}
for all bounded Jordan measurable sets $B \subset \R.$  
\end{theorem}
\begin{remark}
Note that Theorem \ref{llt_del} compared to Theorem \ref{llt}, covers only the exponents $\delta \in \left(0, \frac{1}{2} \right].$ Thus, in some specific cases Theorem \ref{llt} may provide more general results. 
\end{remark}

\textbf{Notations:} We use the Landau notation $f=O(g)$ in some places, meaning that there exists a constant $c$ such that
\begin{align*}
|f(x)| \leq c |g(x)|
\end{align*} 
for all $x$ in a set $X$ which is indicated. The parameter $c$ may depend on further parameters.

\section{Main results}\label{main_results}
Our first theorem states that there is a mod-Gaussian convergence for the lacunary trigonometric series under some additional assumptions. Let $(\akn)_{1 \leq k \leq n}$ be a triangular array. Throughout the article, we use the following notations 
\begin{align}
\ckn := \frac{\akn}{A_n} \label{ckn},\\
d_n := \max_{1 \leq k \leq n} |\akn|. \label{dn}
\end{align}
\begin{theorem} [Mod-Gaussian convergence for lacunary trigonometric series] \label{tryg}
Let $S^{\mathrm{T}}_n(x) = \sum_{k=1}^n \akn \cos(2\pi m_kx)$, $0 \leq x \leq 1$ be the $n$th partial sum of the trigonometric series with  $\left(m_k \right)_{k \in \N}$ satisfying the condition (\ref{lacunary}). Suppose that when $n \to \infty,$ 
\begin{align*}
A_n := &\left( \half \sum_{k=1}^n \akn^2 \right)^{\frac{1}{2}} \to \infty.
\end{align*}
Moreover, we suppose there exists $\varepsilon > 0$ such that
\begin{itemize}[leftmargin=0.15in]
\item for $1<q \leq 2,$  
\begin{align}
n^{1+ \varepsilon} d_n^3 \to 0, \label{dnq1}
\end{align}
\item for $q > 2,$  
\begin{align}
n^{1+ \varepsilon} d_n^4 \to 0. \label{dnq2}
\end{align}
\end{itemize}
Then $S^{\mathrm{T}}_n(x)$ converges mod-Gaussian in the sense of Definition \ref{mod-Gauss-weak}. 
\end{theorem}
\begin{remark}
One can see from the proof (see section \ref{Section:3}) that the assumption of $(m_k)_k$ to be a sequence of integers can be relaxed and Theorem \ref{tryg} also holds  for the sequence of real numbers $(m_k)_k$ that satisfy the lacunarity assumption (\ref{lacunary}).
 \end{remark}
\begin{remark}
Note that in the case when $\akn = \frac{1}{n^\alpha},$ for all $1 \leq k \leq n,$ the conditions of Theorem \ref{tryg} are satisfied 
\begin{itemize}
\item if $\frac{1}{3} < \alpha < \frac{1}{2},$ for $1<q \leq 2,$ 
\item if $\frac{1}{4} < \alpha < \frac{1}{2},$ for $ q > 2.$ 
\end{itemize}
\end{remark}
Theorem \ref{tryg} readily implies the local limit theorem for lacunary trigonometric series.
\begin{theorem}[LLT for lacunary trigonometric series]
Under the assumptions of Theorem \ref{tryg}, one has
\begin{align*}
A_n \left| \left\lbrace  x \in [0,1]; S^{\mathrm{T}}_n(x) \in B \right\rbrace  \right| \to \frac{1}{\sqrt{2 \pi}} |B|,
\end{align*}
for all bounded Jordan measurable sets $B \subset \R.$
\end{theorem}

\begin{theorem}[Mod-Gaussian convergence for lacunary Walsh series]\label{walsh-modG}
Let $S^{\mathrm{W}}_n(x) = \sum_{k=1}^n \akn W_{m_k}(x)$, $0 \leq x \leq 1,$ be the $n$th partial sum of the Walsh series, where $\left(m_k \right)_{k \in \N}$ satisfies the condition (\ref{lacunary}).   Suppose that when $n \to \infty,$
\begin{align*}
&A_n = \left( \sum_{k=1}^n \akn^2 \right)^{\frac{1}{2}} \to \infty.
\end{align*}
Moreover, we suppose that
\begin{itemize}
\item for $q \geq 2,$
\begin{align*}
& \sum_{k=1}^n \akn^4 \to \kappa_4 < \infty,\\
&n d_n^5 \to 0,
\end{align*}
\item for $1 < q < 2,$ there exists $\varepsilon > 0$ such that \begin{align*}
n^{1+\varepsilon} d_n^3 \to 0.
\end{align*}
\end{itemize}
Then $S^{\mathrm{W}}_n(x)$ converges mod-Gaussian in the sense of Definition \ref{mod-Gauss} on $\mathbb{C}$ with parameters $t_n = A_n^2$ and the limiting function 
\begin{align*}
\psi(z)= \begin{cases} &e^{-\frac{z^4}{12} \kappa_4 }, \quad q\geq 2, \\&1, \quad 1<q< 2.\end{cases}
\end{align*}
\end{theorem}
This version of mod-Gaussian convergence immediately implies the extended central limit theorem (see Theorem 4.3.1 in \cite{FMN16}) and moderate deviations (see Theorem 4.2.1 in \cite{FMN16}). 
\begin{theorem}[Extended CLT for lacunary Walsh series] Under the assumptions of Theorem \ref{walsh-modG}, for  $y=o\left(A_n\right)$,
\begin{equation*}
\left| \left\lbrace  x \in [0,1]; \frac{S^{\mathrm{W}}_n(x)}{A_n} \geq y \right\rbrace  \right| = \left(1+o(1)\right) \frac{1}{\sqrt{2 \pi}} \int_{y}^{\infty} e^{-\frac{u^2}{2}} du=\frac{e^{-\frac{y^2}{2}}}{y\sqrt{2\pi}}\left(1+o(1)\right).
\end{equation*}
\end{theorem}

\begin{theorem}[Moderate deviations for lacunary Walsh series ] We assume that the assumptions of Theorem \ref{walsh-modG} are satisfied. 
Then for $y>0$,
\begin{equation*}
\left| \left\lbrace  x \in [0,1]; S^{\mathrm{W}}_n(x) \geq A_n^2 y \right\rbrace  \right|  =\frac{e^{-A_n^2\frac{y^2}{2}}}{yA_n\sqrt{2\pi }}\psi(y)\left(1+o(1)\right),
\end{equation*}
and for $y < 0$,
\begin{equation*}
\left| \left\lbrace  x \in [0,1]; S^{\mathrm{W}}_n(x) \leq A_n^2 y \right\rbrace  \right|  =\frac{e^{-A_n^2\frac{y^2}{2}}}{|y|A_n\sqrt{2\pi }}\psi(y)\left(1+o(1)\right),
\end{equation*}
\end{theorem}
In Section \ref{Section:4} we show that we have a zone of control i.e. the conditions \ref{Z1} and \ref{Z2} are satisfied. As a result, the following two theorems hold.
\begin{theorem}[LLT for lacunary Walsh series]\label{walsh-LLT}
Suppose the assumptions of Theorem \ref{walsh-modG} are satisfied. Let $y \in \R$ and $B$ be a fixed Jordan measurable subset with $|B|>0.$ Then for every exponent $\delta \in \left(0,  \gamma + \frac{1}{2} \right),$ with
\begin{align*}
\gamma = \begin{cases} \frac{1}{10}, \qquad \qquad \mathrm{if} \quad q\geq 2  \\\min \lbrace \frac{\varepsilon}{3} , \frac{1}{3} \rbrace, \quad \mathrm{if} \quad 1 < q < 2 \end{cases} , 
\end{align*} 
one has 
 \begin{equation*}
\lim_{n\to\infty}A_n^{2 \d}\left| \left\lbrace x\in[0,1] ; \frac{ S^{\mathrm{W}}_n(x)}{A_n}-y \in\frac{1}{A_n^{2\d}} B \right\rbrace \right|=\frac{|B|}{\sqrt{2\pi}}.
\end{equation*}
\end{theorem}
\begin{theorem}[Speed of convergence for lacunary Walsh series]\label{walsh-speed}
Let $S^{\mathrm{W}}_n(x)$ be  the $n$th partial sum of the lacunary Walsh series that satisfies the assumptions of Theorem \ref{walsh-modG}, then one has
\begin{align*}
d_{Kol}\left(\frac{S^{\mathrm{W}}_n}{A_n}, \mathcal{N}(0,1)\right)\le C \frac{1}{A_n^{2\gamma+1}}
\end{align*}
where $d_{Kol}(\cdot, \cdot)$ is the Kolmogorov distance, $\gamma$ is specified above and $C$ is a constant (see Theorem 2.16 in \cite{FMN17}).
\end{theorem}

Next we propose another approach to show Theorem \ref{KacT} using martingale theory.
\begin{theorem} \label{CLT_KAC_general}
Let $f:[0,1] \to \R$ be a function in $L^2$ that we extend periodically  to $\R$ (only for notational reasons). We denote by 
\begin{align*}
f_n := \E[f|\mathcal{D}_n], \quad \phi_n := f-f_n,
\end{align*}
where $\mathcal{D}_n$ is the $\sigma$-algebra generated by the intervals $\left(\frac{k}{2^n}, \frac{k+1}{2^n} \right],$ $0 \leq k \leq 2^n-1.$ 
If $\sum_{s \geq 1} \Vert \phi_s \Vert_2 < \infty,$ then 
\begin{align*}
\left| \lbrace x \in [0,1]; \frac{f(x) + \cdots +f\left(2^{n-1} x \right)}{\sqrt{n}} \leq t \rbrace \right| \to \frac{1}{\sigma \sqrt{2 \pi}} \int_{-\infty}^t e^{-\frac{u^2}{2 \sigma^2}} du,
\end{align*}
where $\sigma^2 = \lim_{n \to \infty} \left| \left| \frac{f(x) + \cdots +f\left(2^{n-1} x \right)}{\sqrt{n}}\right| \right|_2^2$ provided this limit is different from zero.
\end{theorem}
\begin{remark}
If $f$ is H\"older continuous with exponent $\beta$ we have that 
\begin{align*}
\Vert \phi_n\Vert_2 \leq \Vert \phi_n\Vert_\infty \leq C2^{-n \beta}.
\end{align*}
The approximation hypotheses of the theorem is therefore satisfied. 
\end{remark}
\begin{remark}
For $f(x) = \sum_{k \geq 1} a_k r_k(x),$ where $\left(r_k(x)\right)_{k \geq 0}$ are Rademacher functions, we have 
\begin{align*}
\Vert \phi_n\Vert_2 = \sqrt{\sum_{k > n} a_k^2}.
\end{align*}
Clearly $ \sum_n \sqrt{\sum_{k > n} a_k^2} < \infty$ implies $\sum_n n^2 a_n^2 < \infty.$ Such a function $f$ is not H\"older continuous.
\end{remark}
\begin{remark}
One can identify the interval $[0,1]$ with $\otimes_{k=1}^\infty \lbrace 0, 1 \rbrace$ equipped with the product measure $\left(\frac{1}{2}, \frac{1}{2}\right).$ The identification is done through the mapping $\left(x_k \right)_k \to \sum_k \frac{x_k}{2^k}.$ Note that the multiplication by $2$ is then a shift.
\end{remark}
The previous theorem can be generalized to infinite product spaces of arbitrary probability spaces, equipped with the shift operator.   
\begin{theorem}\label{CLT_shift} Let $(E, \mathcal{E}, \mu)$ be a probability space. Let $\Omega = \prod_{k \geq 1} E,$ $\mathcal{F}_\infty = \otimes_{k \geq 1} \mathcal{E}$ be the $\sigma$-algebra on the Cartesian product $\Omega$ and $\mathbb{P} = \otimes_{k \geq 1} \mu$ be the product measure.
Furthermore, the projection $\mathrm{pr}_j,  j \in \N$ is defined as
\begin{align*}
\mathrm{pr}_j \quad : \quad \Omega & \to E, \\
(x_1, x_2, \dots) & \to x_j.
\end{align*}
The shift operator $\theta$ defined as
\begin{align*}
\theta : \Omega & \to \Omega \quad \quad \quad \mathrm{satisfies} \quad \P \circ \theta^{-1} = \P \\
(x_1, x_2, \dots) &  \to (x_2, x_3, \dots)
\end{align*}
Suppose $f \in L^2$ such that $\int f d\P = 0.$ We denote by
\begin{align*}
f_r = \E \left[f | \mathcal{F}_r  \right], \quad \phi_r = f - f_r,
\end{align*} 
where $\mathcal{F}_r=\sigma \left(\mathrm{pr}_1, \dots, \mathrm{pr}_r \right)$ and $\mathcal{F}^r=\sigma \left(\mathrm{pr}_{r+1}, \dots,  \right)$ are $\sigma$-algebras.
Suppose $\sum_{r \geq 1} ||\phi_r||_2 < \infty$ then 
\begin{align*}
\P \left[ \frac{1}{\sqrt{n}} \sum_{k=0}^{n-1} f \circ \theta^k \leq t \right]  \to \frac{1}{\sigma \sqrt{2 \pi}} \int_{-\infty}^t e^{-\frac{u^2}{2 \sigma^2}} du,
\end{align*}
with $\sigma^2 = \lim_{n \to \infty} \frac{1}{n} \E \left[\left(f + f \circ \theta + \cdots f \circ \theta^{n-1} \right)^2 \right]$ provided $\sigma^2$ is different from zero.  
\end{theorem}

Next we show that there is a mod-Gaussian convergence in the sense of Definition \ref{mod-Gauss-weak} for H{\"o}lder continuous periodic functions under some additional assumptions when the gap size goes to infinity.
\begin{theorem}[Mod-Gaussian convergence for ``big gaps''] \label{mod-Gauss for Holder} Let $f(x)$ be a measurable function defined $[0,1)$ and extended periodically by setting $f(x) = f(x+1)$ such that  
\begin{align*}
&\left |f(x) - f(y)\right| \leq h |x-y|^\alpha, \quad \alpha>0, x \neq y, \\
&\int_{0}^{1} f(x)dx = 0, \quad \int_{0}^1 f^2(x) dx = 1.
\end{align*}
Let $(m_k)_{k \in \mathbb{N}}$ be an increasing sequence of integers such that $b_k := \frac{m_{k+1}}{m_k} \in \lbrace 2, 3, \dots \rbrace,$ moreover, $b_k \to \infty,$ as $k \to \infty.$
Let $(\akn)_{1 \leq k \leq n}$ be a triangular array such that 
\begin{align*}
&A_n : = \left( \sum_{k=1}^n \akn^2 \right)^{\frac{1}{2}} \to \infty,\\
 &\sum_{k=1}^n |\akn|^3 \to 0,\\
& A_n \sum_{k=1}^n \frac{|\akn|}{b_k^\alpha} \to 0,
\end{align*} 
as $n \to \infty.$ We denote by $S^{\mathrm{H}}_n(x) = \sum_{k=1}^n \akn f\left(m_k x\right),$  then $S^{\mathrm{H}}_n(x)$ converges mod-Gaussian in the sense of Definition \ref{mod-Gauss-weak}.
\end{theorem}
As before, the local limit theorem is a direct consequence of the theorem above.
\begin{theorem}[LLT for ``big gaps'']
Under the assumptions of Theorem \ref{mod-Gauss for Holder}, one has
\begin{align*}
A_n \left| \left\lbrace  x \in [0,1];S^{\mathrm{H}}_n(x) \in B \right\rbrace  \right| \to \frac{1}{\sqrt{2 \pi}} |B|,
\end{align*}
for all bounded Jordan measurable sets $B \subset \R.$
\end{theorem}
Next we discuss the function $f(x) = x -\lfloor  x \rfloor - \half ,$ the first Bernoulli polynomial. In \cite{K38} a central limit theorem has been shown for this function and we are going to show that there is also mod-Gaussian convergence.
\begin{theorem}[Mod-Gaussian convergence for $f(x) = x -\lfloor  x \rfloor - \half $]\label{mod-Gauss specf} Let $S^{f}_n(x) = \frac{1}{n^{1/4}} \sum_{k=1}^n  f\left(2^k x\right),$ with $f(x) = x -\lfloor  x \rfloor - \half. $ Then $\frac{S^{f}_n(x)}{n^{1/4}}$ converges mod-Gaussian in the sense of Definition \ref{mod-Gauss} on $\mathbb{C}$ with parameters $t_n = \frac{\sqrt{n}}{4}$ and the limiting function  $ \psi(z)= e^{-\frac{z^4}{192} }.$
\end{theorem}
As before,  the extended central limit theorem (see Theorem 4.3.1 in \cite{FMN16}) and moderate deviations (see Theorem 4.2.1 in \cite{FMN16}) are immediate consequences. 
\begin{theorem}[Extended CLT for $f(x) = x -\lfloor  x \rfloor - \half $] For  $y=o\left(n^{1/4}\right)$,
\begin{equation*}
\left| \left\lbrace  x \in [0,1]; \frac{2 S^{f}_n(x)}{\sqrt{n}} \geq y \right\rbrace  \right| \to \left(1+o(1)\right) \frac{1}{\sqrt{2 \pi}} \int_{y}^{\infty} e^{-\frac{u^2}{2}} du=\frac{e^{-\frac{y^2}{2}}}{y\sqrt{2\pi}}\left(1+o(1)\right).
\end{equation*}
\end{theorem}

\begin{theorem}[Moderate deviations for $f(x) = x -\lfloor  x \rfloor - \half $ ] 
For $y>0$,
\begin{equation*}
\left| \left\lbrace  x \in [0,1]; S^{f}_n(x) \geq \frac{n^{3/4} y}{4} \right\rbrace  \right|  =\frac{e^{-\frac{\sqrt{n}}{4}\frac{y^2}{2}}}{y\sqrt{\frac{\pi \sqrt{n}}{2} }}\psi(y)\left(1+o(1)\right),
\end{equation*}
and for $y < 0$,
\begin{equation*}
\left| \left\lbrace  x \in [0,1]; S^{f}_n(x) \leq \frac{n^{3/4} y}{4} \right\rbrace  \right|  =\frac{e^{-\frac{\sqrt{n}}{4}\frac{y^2}{2}}}{|y|\sqrt{\frac{\pi \sqrt{n}}{2} }}\psi(y)\left(1+o(1)\right).
\end{equation*}
\end{theorem}
In Section \ref{Section:5} we show that we have a zone of control i.e. the conditions \ref{Z1} and \ref{Z2} are satisfied. As a result, the following two theorems hold.
\begin{theorem}[LLT for $f(x) = x -\lfloor  x \rfloor - \half $]\label{spec-LLT}
Let $y \in \R$ and $B$ be a fixed Jordan measurable subset with $|B|>0.$ Then for every exponent $\delta \in  \left(0,\frac{13}{24}\right),$
one has 
 \begin{equation*}
\lim_{n\to\infty}\left(\frac{\sqrt{n}}{4}\right)^{\d}\left| \left \lbrace x \in [0,1]; \frac{2 S_n^f(x)}{\sqrt{n}}-y \in\left(\frac{4}{\sqrt{n}}\right)^\d B \right \rbrace \right|=\frac{|B|}{\sqrt{2\pi}}.
\end{equation*}
\end{theorem}
\begin{theorem}[Speed of convergence for $f(x) = x -\lfloor  x \rfloor - \half $]\label{spec-speed}
One has
\begin{align*}
d_{Kol}\left(\frac{2S^{f}_n}{\sqrt{n}}, \mathcal{N}(0,1)\right)\le C \left(\frac{4}{\sqrt{n}} \right)^{\frac{13}{24}},
\end{align*}
where  $C$ is a constant that can be calculated explicitly (see Theorem 2.16 in \cite{FMN17}).
\end{theorem}
\section{Proof of Theorem \ref{tryg}} \label{Section:3}
\begin{lemma} \label{lemma:3.1}
Let $(m_k)_{k\geq 1}$ be a sequence satisfying the condition (\ref{lacunary}) with $1 < q \leq 2.$ We denote by  $C_r(l,p,q,n)$ the number of solutions of the equation
\begin{align}
\varepsilon_1 m_{k_1} \pm \varepsilon_2 m_{k_2} \pm \cdots \pm \varepsilon_l m_{k_l} = A, \label{ptryg-equation} 
\end{align} 
where $1 \leq k_l < \dots < k_1 \leq n, l \in \mathbb{N}$ and $A \in \mathbb{Z}.$ Moreover, 
\begin{align*}
 \varepsilon_i \in \lbrace 1, 2, 3, \dots, r \rbrace
\end{align*} 
for all $i \geq 1, r \geq 1$ and $p \leq l$ is the number of $\varepsilon_i$'s that are different from $1.$ We claim that $C_r(l, p, q,n)$ is bounded by $\left(8 n \log_q(rl) \log_q \left(\frac{2r^2lq}{(q-1)^2}\right) \right)^{\frac{l+p}{3}}.$ 
\end{lemma}
\begin{remark}
In the case $r=1,$ Lemma \ref{lemma:3.1} states that the number of solutions of the equation
\begin{align}
m_{k_1} \pm  m_{k_2} \pm \cdots \pm m_{k_l} = A \label{tryg-equation} 
\end{align}
is at most $\left(8n \log_q(l) \log_q \left(\frac{2lq}{(q-1)^2}\right) \right)^{\frac{l}{3}}.$ For example, if the sequence $(m_k)_{k \in \N}$ is taken to be
\begin{align}\label{example}
\begin{cases} m_{2k} &= \quad 2^k,\\
m_{2k+1} &=  \quad 2^k + 2^{k-1}\\
\end{cases},
\end{align}
then it is a lacunary sequence with $q = 4/3.$ Furthermore, the number of solutions of Equation (\ref{tryg-equation}), $C_1(l,0,q,n),$ with $A=0$ is at least of order $n^{l/3}.$ Therefore, it is not possible to improve the factor $n^{l/3}$ appearing in the statement of Lemma \ref{lemma:3.1}. 

In the same vein, for the sequence $(m_k)_{k \in \N}$ constructed in (\ref{example}), taking $p=\frac{l}{2}$ and $r \geq 3,$ we observe that $C_r(l,l/2,q,n),$ is of order $n^{l/2}.$ Thus, the factor $n^{\frac{l+p}{3}}$ appearing in Lemma \ref{lemma:3.1} cannot be improved either.
\end{remark}
\begin{remark}
In \cite{E62}, it has been shown under more general lacunarity assumptions (\ref{Erdos_lac}) that the number of solutions of the equation 
\begin{align*}
m_{k_1} \pm  m_{k_2} \pm \cdots \pm m_{k_l} = A
\end{align*}
is at most $o(n^{l/2}),$ where it's additionally allowed $k_1 = \dots = k_l$ and the trivial solutions are excluded. 
\end{remark}

\begin{proof}[Proof of Lemma \ref{lemma:3.1}] Without loss of generality we assume $\varepsilon_1 =1.$ If $\varepsilon_1 \neq 1,$ we divide both sides of Equation \ref{tryg-equation} by $\varepsilon_1$ and instead of considering the cases $\varepsilon_2 = 1$ and $\varepsilon_2 \neq 1,$ we would discuss the cases $\varepsilon_2 \leq 1$ and $\varepsilon_2 > 1,$ respectively.

To prove the claim we use an induction on $l.$ We first discuss the case $l=2.$ We distinguish two cases.
\begin{itemize}[leftmargin=0.15in]
\item $\varepsilon_2 = 1:$ Note
\begin{align*}
 m_{k_1} \left(1+ \frac{1}{q}\right) \geq m_{k_1} \pm m_{k_2} =  A \geq  m_{k_1} \left(1-\frac{1}{q} \right)
\end{align*}
Thus $ A \frac{q}{q-1} \geq m_{k_1} \geq A \frac{q}{q+1},$ which with the lacunarity condition (\ref{lacunary}) implies that we can choose $m_{k_1}$ at most in finite $\log_q\frac{q+1}{q-1} $ ways. It remains to show $\log_q\left(\frac{q+1}{q-1}\right)  \leq \left(8 n \log_q(2) \log_q \left(\frac{4q}{(q-1)^2}\right) \right)^{2/3},$ which is equivalent to $\log \left(\frac{q+1}{q-1} \right)\, \log^{1/3} \left( q \right)  \leq \left(16 \log(2)\log \left(\frac{4q}{(q-1)^2}\right) \right)^{2/3}.$ \, Since $\log \left(\frac{q+1}{q-1} \right) \log \left(q\right) < 2$ and  $\frac{q+1}{q-1} <\frac{4q}{(q-1)^2}  $ for $1<q \leq 2,$ we deduce 
\begin{align*}
\log \left(\frac{q+1}{q-1} \right)\, \log^{1/3} \left( q \right) < 2^{1/3}\log^{2/3} \left(\frac{q+1}{q-1} \right) <   \left(16 \log(2)\log \left(\frac{4q}{(q-1)^2}\right) \right)^{2/3}.
\end{align*}
\item $\varepsilon_2 \neq 1:$ We intend to show that the number of solutions is  at most $8n \log_q(2r) \log_q \left(\frac{4r^2q}{(q-1)^2}\right) .$  We can choose $m_{k_1}$ at most in $n$ ways and we intend to show that $m_{k_2}$ can be chosen at most in $\log_q(r/2)$ ways. Since $m_{k_1}$ has been already chosen, we have $\varepsilon_2 m_{k_2} = \pm \left(A - m_{k_1} \right). $ Let $\varepsilon'_2, m_{k'_2}$ be another pair satisfying $\varepsilon'_2 m_{k'_2} =\pm \left( A - m_{k_1} \right) $ and without loss of generality we assume $m_{k_2} > m_{k'_2}.$ As a result, we obtain
\begin{align*}
1=\frac{\varepsilon_2 m_{k_2}}{\varepsilon'_2 m_{k'_2}} \geq \frac{2q^{k_2 - k'_2}}{r}.
\end{align*}
Therefore, when $m_{k_1}$ has been already chosen, $m_{k_2}$ can be chosen at most in $\log_q(r/2)$ ways, resulting at most $8n \log_q(r/2) < 8n \log_q(2r) \log_q \left(\frac{4r^2q}{(q-1)^2}\right) $ solutions.
\end{itemize}
We now treat the general case. We suppose that the claim is true for all $l^\prime < l$ and we aim to show that it holds for $l.$ Moreover, we assume that the number of $\varepsilon_i$ different from $1$ is equal to $p.$ We first distinguish two cases.
\begin{itemize}[leftmargin=0.15in]
\item $\varepsilon_2 \neq  1:$ Two further possibilities need to be discussed.
\begin{itemize}[leftmargin=0.15in] 
\item $\frac{m_{k_1}}{m_{k_2}} \leq l r :$ We choose $m_{k_1}$ in $n$ ways. Therefore $m_{k_2}$ can be chosen in  $\log_q \left( lr\right)$ ways, resulting $n \log_q \left( lr\right)$ total possibilities for $m_{k_1}$ and $m_{k_2}.$
\item $\frac{m_{k_1}}{m_{k_2}} > lr:$ We have 
\begin{align}\label{gen_bounds}
2 r  m_{k_1} > m_{k_1} \left( 1+ \frac{l-1}{l}\right)> m_{k_1} + \cdots + \varepsilon_l m_{k_l} = A > m_{k_1} \left( 1 - \frac{l-1}{l}\right) = \frac{m_{k_1}}{l},
\end{align}
which implies that $m_{k_1}$ can be chosen at most in finite $\log_q \left( 2rl\right)$ ways. Choosing $m_{k_2}$ in $n$ ways gives us at most $n \log_q \left( 2rl\right)$ possibilities to choose $m_{k_1}$ and $m_{k_2}.$

\end{itemize}
We conclude that in the case $\varepsilon_2 \neq 1,$ $m_{k_1}$ and $m_{k_2}$ can be chosen at most in $2 n \log_q(2rl)$ ways which  gives the following bound for the number of solutions of Equation (\ref{ptryg-equation}), in case of  $\varepsilon_2 \neq 1,$
\begin{align}
2n \, \log_q (2rl) \, C_r(l-2,p-1,q,n)  .\label{b:1} 
\end{align}
\item $\varepsilon_2 = 1:$ 
We consider two cases.
\begin{itemize}[leftmargin=0.15in]
\item $\frac{m_{k_1}}{m_{k_{2}}} > lr:$  Note that the bounds (\ref{gen_bounds}) hold also in this case.  Thus, $m_{k_1}$ can be chosen in $\log_q(2lr)$ ways.  We discuss two further cases. 
\begin{itemize}[leftmargin=0.15in]
\item $\frac{m_{k_2}}{m_{k_{3}}} > lr$: Similar to $m_{k_1}$ we have $2 r m_{k_2}  > \pm \left(A - m_{k_1} \right) > \frac{2 m_{k_2}}{l}, $ hence if $m_{k_1}$ has been already chosen, $m_{k_2}$ can be chosen at most in $\log_q( r l)$ ways.  Moreover, we can choose $m_{k_3}$ at most in $n$ ways. 
\item $\frac{m_{k_2}}{m_{k_{3}}} \leq l r$: In this case $m_{k_2}$ can be chosen in $n$ ways and $m_{k_3}$ in $\log_q(l r)$ ways.
 \end{itemize}
 Hence, the following bound holds for the number of solutions of Equation (\ref{ptryg-equation}) for $\frac{m_{k_1}}{m_{k_2}} > lr$ and $\varepsilon_2=1.$
\begin{align}
2 n \log_q\left( 2 r l\right) \, \log_q\left( r l\right) \, C_r(l-3, p,q, n) \label{b:2}
\end{align} 
\item $\frac{m_{k_1}}{m_{k_2}} \leq l r$: We discuss two further possibilities.
\begin{itemize}[leftmargin=0.15in]
\item $\frac{m_{k_2}}{m_{k_3}} < \frac{2l r}{q-1}:$  We can choose $m_{k_1}$ at most in $n$ ways. Moreover,  $q^{k_1-k_2}\leq \frac{m_{k_1}}{m_{k_2}} \leq lr$, thus $k_1-k_2 \leq \log_q (lr).$ So if $m_{k_1}$ has been already chosen, $m_{k_2}$ can be chosen in $\log_q (lr)$ ways. In the same vein, we obtain that $m_{k_3}$ can be chosen in $\log_q \left(\frac{2l r}{q-1} \right)$ ways. We conclude that $m_{k_1}, m_{k_2}, m_{k_3}$ can be chosen at most in $n \log_q\left(l r \right) \log_q \left(\frac{2l r}{q-1} \right)$ ways.
\item $\frac{m_{k_2}}{m_{k_3}} \geq \frac{2l r}{q-1}:$ Note that using Equation (\ref{ptryg-equation}), we get
\begin{align*}
A \geq m_{k_2} (q-1) - \varepsilon_3 m_{k_3} - \cdots - \varepsilon_l m_{k_l} > m_{k_2} (q-1) - lr m_{k_3} > \frac{m_{k_2} (q-1)}{2} > \frac{m_{k_1} (q-1)}{2 lr}.
\end{align*} 
 On the other hand, 
\begin{align*}
A \leq r m_{k_1} \left(1 + \frac{1}{q} + \cdots + \frac{1}{q^{l-1}}\right) < m_{k_1} \frac{rq}{q-1}.
\end{align*}
We deduce there are at most $\log_q\left( \frac{2lqr^2}{(q-1)^2}\right)$ ways to choose $m_{k_1}$ and $\log_q(lr)$ ways to choose $m_{k_2},$ as $\frac{m_{k_1}}{m_{k_2}} \leq lr.$  Finally, $m_{k_3}$ can be chosen at most in $n$ ways. Thus, $m_{k_1}, m_{k_2}, m_{k_3}$ can be chosen in $n \log_q\left(lr \right)\log_q\left( \frac{2lqr^2}{(q-1)^2}\right)$ ways.  
\end{itemize}
We obtain that the number of solutions of Equation (\ref{ptryg-equation}) for $\frac{m_{k_1}}{m_{k_{2}}} \leq lr$ and $\varepsilon_2 = 1$ is bounded by
\begin{align}
2 n \log_q\left(lr \right)\log_q\left( \frac{2lqr^2}{(q-1)^2}\right) C_r(l-3, p, q, n) \label{b:3}.
\end{align} 
\end{itemize}
\end{itemize}
Summing up the bounds (\ref{b:1}), (\ref{b:2}) and (\ref{b:3}) and using the induction hypothesis, we conclude 
\begin{align*}
C_r(l,p,q,n) <  \left(8 n \log_q(rl) \log_q \left(\frac{2r^2lq}{(q-1)^2}\right) \right)^{\frac{l+p}{3}}
\end{align*} 
and the proof of the lemma follows.
\end{proof}
Next we are interested in obtaining a result similar to Lemma \ref{lemma:3.1} for $q > 2.$ In fact, using Lemma \ref{lemma:3.1} we could show that the number of solutions for $q > 2$ is again of order $n^{\frac{l+p}{3}}.$ However, we aim to show that for $q > 2,$ this estimate can be improved allowing us to impose more general assumptions in Theorem \ref{tryg} for $q > 2.$

\begin{lemma} \label{lemma:3.2}
Let $(m_k)_{k\geq 1}$ be a sequence satisfying the lacunarity condition (\ref{lacunary}) with $q > 2.$ We denote by  $C_r(l,p_2, p_3,q,n)$ the number of solutions of the equation
\begin{align}
\varepsilon_1 m_{k_1} \pm \varepsilon_2 m_{k_2} \pm \cdots \pm \varepsilon_l m_{k_l} = A, \label{pqtryg-equation} 
\end{align} 
where $1 \leq k_l < \dots < k_1 \leq n, l \in \mathbb{N}$ and $A \in \mathbb{Z}.$ Moreover, 
\begin{align*}
 \varepsilon_i \in \lbrace 1, 2, 3, \dots, r \rbrace
\end{align*} 
for all $i \geq 1.$ Note that $p_2$ and $p_3$  are the number of $\varepsilon_i$'s such that $\varepsilon_i =2$ and $\varepsilon_i \geq 3,$ respectively. We claim that $C_r(l, p_2, p_3, q,n)$ is bounded by $\left(20 n  \log_q\left( 2l r \right) \log_q\left( \frac{ql r}{q-2} \right) \log_q\left( \frac{4l^2 q^2 r^3}{q-2} \right)  \right) ^{\frac{l}{4}+\frac{p_2}{4}+\frac{p_3}{2}}.$ 
\end{lemma}
\begin{remark}
When $r=1$ and $A=0,$ one has $C_1(l, 0,0, q,n) = 0,$ since
\begin{align*}
&0=m_{k_1} \pm  m_{k_2} \pm \cdots \pm  m_{k_l} >m_{k_1} \left( 1 - \frac{1}{2} - \frac{1}{4} - \cdots - \frac{1}{2^{l-1}}\right) > 0.
\end{align*}
As a result there are no $m_{k_1}, \dots, m_{k_l}$ satisfying (\ref{pqtryg-equation}) for $A=0, r=1, q > 2.$ 
\end{remark}
\begin{remark}
Taking $m_k = 3^k$ for all $k \in\N$ we easily see that the factor $n^{\frac{l}{4}+\frac{p_2}{4}+\frac{p_3}{2}}$ in Lemma \ref{lemma:3.2} cannot be improved.
\end{remark}

\begin{proof}[Proof of Lemma \ref{lemma:3.2}] Without loss of generality we assume $\varepsilon_1 = 1.$ To prove the claim we use an induction on $l.$ We first discuss the case $l=2.$ We distinguish the following cases.
\begin{itemize}[leftmargin=0.15in]
\item $\varepsilon_2 = 1,2:$ Note
\begin{align*}
 m_{k_1} \left(1+ \frac{2}{q}\right) \geq m_{k_1} \pm m_{k_2} =  A \geq  m_{k_1} \left(1-\frac{2}{q} \right).
\end{align*}
Thus, $ A \frac{q}{q-2} \geq m_{k_1} \geq A \frac{q}{q+2},$ which with the lacunarity condition (\ref{lacunary}) implies that we can choose $m_{k_1}$ at most in finite $\log_q\frac{q+2}{q-2} \leq \left(40 \log_q\left( 4 \right) \log_q\left( \frac{2q}{q-2} \right) \log_q\left( \frac{16 q^2 }{q-2} \right)  \right) ^{\frac{1}{2}}$ ways. 
\item $\varepsilon_2 \geq 3:$ Similar to the proof of Lemma \ref{lemma:3.1}, we can choose $m_{k_1}$ in $n$ ways and $m_{k_2}$ in $\log_q(r/3)$ ways, giving in total at most $n \log_q(r/3)  $ solutions.
\end{itemize}
We now consider the general case. We suppose that the claim is true for all $l^\prime < l$ and we intend to show that it holds for $l.$ We first distinguish the following cases.
\begin{itemize}[leftmargin=0.15in]
\item $\varepsilon_2 \geq  3:$ The bounds (\ref{b:1}) hold also in this case i.e. we get that the number of solutions of Equation (\ref{pqtryg-equation}) is at most
\begin{align}
2n \, \log_q (2rl) \, C_r(l-2,p_2, p_3-1,q,n)  .\label{b:2.0}
\end{align}
\item $\varepsilon_2 = 2:$ We examine two possibilities.
\begin{itemize}[leftmargin=0.15in]
\item $\frac{m_{k_1}}{m_{k_2}} > lr:$ Here it is possible to apply the bounds (\ref{gen_bounds}) and deduce that $m_{k_1}$ can be chosen in $\log_q(2lr)$ ways.  Further considering the cases $\frac{m_{k_2}}{m_{k_{3}}} > lr$ and $\frac{m_{k_2}}{m_{k_{3}}} \leq lr,$ we end up having 
\begin{align*}
2 n \log_q(2lr) \log_q(lr) C_r(l-3,p_2-1,p_3,q,n)
\end{align*}
possible solutions.
\item $\frac{m_{k_1}}{m_{k_2}} \leq lr:$ We consider the cases:
\begin{itemize}[leftmargin=0.15in]
\item $\frac{m_{k_2}}{m_{k_3}} \leq \frac{2lr}{q-2}:$ Here $m_{k_1}, m_{k_2}$ and $m_{k_3}$ can be chosen in $n, \log_q(lr)$ and $\log_q \left( \frac{2lr}{q-2}\right)$ ways,  respectively, which yields to the following bound of the possible solutions: 
\begin{align*}
n \log_q(lr) \log_q \left( \frac{2lr}{q-2}\right) C_r(l-3,p_2-1, p_3,q,n).
\end{align*}
\item $\frac{m_{k_2}}{m_{k_3}} > \frac{2lr}{q-2}:$ Using Equation (\ref{pqtryg-equation}), we have
\begin{align*}
m_{k_1} \frac{rq}{q-1} >A > m_{k_2} \left( q-2\right) - lrm_{k_3} > \frac{m_{k_2} \left( q-2\right)}{2} > \frac{m_{k_1} \left( q-2\right)}{2lr} ,
\end{align*}
Thus, $m_{k_1}, m_{k_2},$ and $ m_{k_3}$ can be chosen at most in $\log_q\left(\frac{2lr^2q}{(q-1)(q-2)} \right), \log_q \left(lr \right)$ and $n$ ways, respectively, resulting at most
\begin{align*}
n \log_q\left(\frac{2lr^2q}{(q-1)(q-2)} \right) \log_q \left(lr \right) C_r(l-3,p_2-1,p_3, q,n) 
\end{align*}
possible solutions.
\end{itemize}
\end{itemize}
We derive that the number of solutions satisfying Equation (\ref{pqtryg-equation}) for $\varepsilon_2 = 2$ is at most
\begin{align}
4n \log_q\left(\frac{2lr^2q}{(q-1)(q-2)} \right) \log_q \left(lr \right) C_r(l-3,p_2-1,p_3, q,n) \label{b:2.1}
\end{align}
\item $\varepsilon_2 = 1, \varepsilon_3 = 2:$ The strategy used in the previous case, works here as well. The only difference is that instead of considering the cases $\frac{m_{k_2}}{m_{k_3}} > \frac{2lr}{q-2}$ and $\frac{m_{k_2}}{m_{k_3}} \leq \frac{2lr}{q-2},$ we discuss, respectively, $\frac{m_{k_2}}{m_{k_3}} > \frac{2lr}{q-1}$ and $\frac{m_{k_2}}{m_{k_3}} \leq \frac{2lr}{q-1}$. As a result, the bound (\ref{b:2.1}) becomes in this case
\begin{align}
4n \log_q\left(\frac{2lr^2q}{(q-1)^2} \right) \log_q \left(lr \right) C_r(l-3,p_2-1,p_3, q,n) \label{b:2.2}
\end{align}
\item $\varepsilon_2 = 1, \varepsilon_3 \geq 3:$ Here the bounds (\ref{b:2}) and (\ref{b:3}) are applicable giving us at most
\begin{align}
2 n \log_q\left( 2 r l\right) \, \log_q\left( r l\right) \, C_r(l-3, p_2, p_3-1,q, n) \label{b:2.3}
\end{align} 
possible solutions of Equation (\ref{pqtryg-equation}) for $\frac{m_{k_1}}{m_{k_2}} > lr$ and
\begin{align}
2 n \log_q\left(lr \right)\log_q\left( \frac{2lqr^2}{(q-1)^2}\right) C_r(l-3, p_2, p_3-1, q, n) \label{b:2.4}
\end{align}
solutions of Equation (\ref{pqtryg-equation}) for $\frac{m_{k_1}}{m_{k_2}} \leq lr.$ 
\item $\varepsilon_2 = 1, \varepsilon_3 = 1:$ This case we need to treat carefully. The idea is to show that the number of $m_{k_1}, m_{k_2}, m_{k_3}, m_{k_4}$ satisfying Equation (\ref{pqtryg-equation}) is of order $n.$ Thus, we discuss two further possibilities.
\begin{itemize}[leftmargin=0.15in]
\item $\frac{m_{k_1}}{m_{k_{2}}} > lr:$  Note that the bounds (\ref{gen_bounds}) apply in this case as well.  Thus, $m_{k_1}$ can be chosen in $\log_q(2lr)$ ways.  We discuss two further cases. 
\begin{itemize}[leftmargin=0.15in]
\item $\frac{m_{k_2}}{m_{k_{3}}} > lr$: Similar to $m_{k_1}$ we have $2 r m_{k_2}  > \pm \left(A - m_{k_1} \right) > \frac{2 m_{k_2}}{l}, $ hence if $m_{k_1}$ has been already chosen, $m_{k_2}$ can be chosen at most in $\log_q( r l)$ ways. 
\begin{itemize}[leftmargin=0.15in]
\item $\frac{m_{k_3}}{m_{k_4}} > lr:$ In the same manner, $m_{k_3}$ can be chosen at most in $\log_q( r l)$ ways and $m_{k_4}$ in $n$ ways.
\item$\frac{m_{k_3}}{m_{k_{4}}} \leq lr:$ We choose $m_{k_3}$ in $n$ ways and $m_{k_4}$ in $\log_q(lr)$ ways.
\end{itemize} 
The observations above lead to the following bound for the number of solutions of (\ref{pqtryg-equation}), when $\varepsilon_2 = 1, \varepsilon_3 = 1, \frac{m_{k_1}}{m_{k_{2}}} > lr, \frac{m_{k_2}}{m_{k_{3}}} > lr:$
\begin{align}
2n\log_q(2lr) \log^2_q( r l) \, C_r(l-4,p_2,p_3,q,n). \label{b:2.5}
\end{align}
\item $\frac{m_{k_2}}{m_{k_{3}}} \leq l r$: 
\begin{itemize}[leftmargin=0.15in]
\item $\frac{m_{k_3}}{m_{k_4}} > \frac{2lr}{q-1}:$ Using Equation (\ref{pqtryg-equation}), we have
\begin{align*}
\pm (A-m_{k_1}) = m_{k_2} \pm m_{k_3} \pm \varepsilon_4 m_{k_4} \pm \cdots \pm \varepsilon_l m_{k_l} > m_{k_3} \left( q-1\right) - lrm_{k_4} > \frac{m_{k_3} \left( q-1\right)}{2},
\end{align*}
leading to $\pm (A-m_{k_1})  > \frac{m_{k_2} \left( q-1\right)}{2lr}.$ On the other hand, again using Equation (\ref{pqtryg-equation}), we have
\begin{align*}
\pm (A-m_{k_1}) < m_{k_2} \frac{qr}{q-1}.
\end{align*}
Therefore, $m_{k_2}$ can be chosen at most in $\log_q \left(\frac{2lq r^2}{(q-1)^2} \right).$ Finally, $m_{k_3}$ and $m_{k_4}$ can be chosen in $\log_q(lr)$ and $n$ ways, respectively.
\item$\frac{m_{k_3}}{m_{k_4}} \leq \frac{2lr}{q-1}:$ We choose $m_{k_2}, m_{k_3}$ and $m_{k_4}$ in $n, \log_q(lr)$ and $\log_q\left( \frac{2lr}{q-1} \right)$ ways, respectively. 
\end{itemize} 
Thus, the two cases above give the following bound for the number of solutions of (\ref{pqtryg-equation}), when $\varepsilon_2 = 1, \varepsilon_3 = 1, \frac{m_{k_1}}{m_{k_{2}}} > lr, \frac{m_{k_2}}{m_{k_{3}}} \leq lr:$
\begin{align}
n \log_q(2lr) \log_q( r l) \left( \log_q\left( \frac{2lr}{q-1} \right)+ \log_q \left(\frac{2lq r^2}{(q-1)^2} \right)\right) \, C_r(l-4,p_2,p_3,q,n). \label{b:2.6}
\end{align}
 \end{itemize}
 \item $\frac{m_{k_1}}{m_{k_2}} \leq l r$: We discuss two further possibilities.
\begin{itemize}[leftmargin=0.15in]
\item $\frac{m_{k_2}}{m_{k_3}} \leq \frac{2lr}{q-1}:$  
\begin{itemize}[leftmargin=0.15in]
\item $\frac{m_{k_3}}{m_{k_4}} > \frac{2lr}{q^2-q-1}:$ We have
\begin{align*}
A > m_{k_3} \left(q^2 - q -1\right) - lrm_{k_4}  >\frac{m_{k_3} \left(q^2 - q -1\right) }{2} >  \frac{m_{k_1} \left(q^2 - q -1\right) (q-1)}{4 l^2 r^2}. 
\end{align*}
Thus, $m_{k_1}$ can be chosen in $\log_q\left(\frac{4ql^2r^3}{(q-1)^2(q^2-q-1)}\right)$ ways. Moreover, $m_{k_2}, m_{k_3}$ and $m_{k_4}$ can be chosen in $\log_q(lr), \log_q\left(\frac{2lr}{q-1}\right)$ and $n$ ways, respectively. 
\item $\frac{m_{k_3}}{m_{k_4}} \leq \frac{2lr}{q^2-q-1}:$ We choose $m_{k_1}, m_{k_2}, m_{k_3}$ and $m_{k_4},$ respectively, in $n, \log_q(lr), \log_q\left(\frac{2lr}{q-1}\right)$ and $\log_q\left(\frac{2lr}{q^2-q-1} \right)$ ways.
\end{itemize}
Summing up two above cases, we deduce that the number of solutions of Equation (\ref{pqtryg-equation}) for $\frac{m_{k_1}}{m_{k_2}} \leq lr, \frac{m_{k_2}}{m_{k_3}} \leq \frac{2lr}{q-1}$ is bounded by
\begin{align}
n \log_q(lr) \log_q\left(\frac{2 l r}{q-1}\right)  \left( \log_q\left(\frac{2 l r}{q^2-q-1}\right) + \log_q \left( \frac{4 l^2 r^3q}{(q^2-q-1)(q-1)^2} \right) \right) \, C_r(l-4,p_2,p_3,q,n) \label{b:2.7}
\end{align}
\item $\frac{m_{k_2}}{m_{k_3}} > \frac{2lr}{q-1}:$ We note
\begin{align*}
A > m_{k_2} (q-1) - lr m_{k_3} > \frac{m_{k_2}(q-1)}{2} >  \frac{m_{k_1}(q-1)}{2lr}.
\end{align*}
We conclude that $m_{k_1}$ can be chosen at most in $\log_q \left( \frac{2lr^2q}{(q-1)^2}\right)$ ways. Moreover,  we choose $m_{k_2}$ in $\log_q(lr)$ ways. Finally, $m_{k_3}$ and $m_{k_4}$ can be chosen at most in $2n \log_q(lr)$ ways, giving us at most 
\begin{align}
2n  \log^2_q(lr) \log_q \left( \frac{2lr^2q}{(q-1)^2}\right)   C(l-4, p_2, p_3, q, n) \label{b:2.8}
\end{align}
possible solutions.
\end{itemize}
\end{itemize}
\end{itemize}
Summing up the bounds (\ref{b:2.0})-(\ref{b:2.8}) we conclude 
\begin{align*}
C_r(l,p_2,p_3,q,n) <  \left(20 n  \log_q\left( \frac{2l rq}{q-2} \right) \log_q\left( ql r \right) \log_q\left( \frac{4l^2 q^2 r^2}{q-2} \right)  \right) ^{\frac{l}{4}+\frac{p_2}{4}+\frac{p_3}{2}}
\end{align*} 
and the proof of the lemma follows.
\end{proof}
\begin{proof} [Proof of Theorem \ref{tryg}]
 Let $\varphi_n(i \l)$ be a characteristic function of $S^{\mathrm{T}}_n$, i.e.
\begin{align*}
\varphi_n(i \l) = \int_0^{1} e^{ i \l S^{\mathrm{T}}_n(x)} dx = \int_0^{1} e^{ i \l \sum_{k=1}^n  \akn \cos(2 \pi m_kx) } dx.
\end{align*}
To show the mod-Gaussian convergence in the sense of Definition \ref{mod-Gauss-weak}, we need to show that the conditions \ref{H1} and \ref{H2} are satisfied. Note Condition \ref{H1} is satisfied by Theorem \ref{SZ}, it remains to show the validity of Condition \ref{H2}. i.e. $\varphi_n \left(\frac{i \l}{A_n} \right) \mathbbm{1}_{ \lbrace |\l| \leq A_n K\rbrace }$ is uniformly integrable for all $K \geq 0$. 
It is equivalent to show that as $n$ goes to infinity,
\begin{align*}
 \varphi_n \left( \frac{i\l}{A_n}\right) \mathbbm{1}_{ \lbrace |\l| \leq A_n K\rbrace }  \to e^{- \frac{\l^2}{2}}
\end{align*}
 in $L^1$ for all $K \geq 0.$
Following the notation (\ref{ckn}), we have
\begin{align*} 
\varphi_n \left( \frac{i \l}{A_n}\right) &= \int_0^1 e^{ i \l \sum_{k=1}^n \ckn \cos(2 \pi m_k x) } dx  = \frac{1}{2 \pi } \int_0^{2 \pi} e^{ i \l \sum_{k=1}^n \ckn \cos( m_k x) } dx.
\end{align*}
For simplicity, we denote by 
\begin{align*}
B_{k,n}(\l, x) &:= i \l \ckn \cos( m_k x) - \frac{\l^2  \ckn^2}{2} \cos^2( m_k x) - \frac{i \l^3 \ckn^3}{6}  \cos^3( m_k x) + \frac{ \l^4 \ckn^4}{24}  \cos^4( m_k x).
\end{align*} Therefore, we can write
\begin{align}
\varphi_n \left( \frac{i \l}{A_n}\right) &= \frac{1}{2 \pi }\int_0^{2 \pi} \prod_{k=1}^n \left( 1+B_{k,n}(\l, x)\right)   dx   \label{rel:1}\\ 
&+ \frac{1}{2 \pi }\int_0^{2 \pi} \left(e^{    \gamma_n(\l,x) } -1\right)\prod_{k=1}^n \left( 1+B_{k,n}(\l, x)\right)  dx \label{rel:2},
\end{align}
with $\gamma_n(\l,x) = \sum_{k=1}^n \left(i \l \ckn \cos(m_k x) - \log\left(1 + B_{k,n} (\l,x)\right) \right).$ Moreover, 
\begin{align*}
\gamma_n(\l,x) = O\left(|\l|^5 \sum_{k=1}^n |\ckn|^5 \right).
\end{align*}
In addition, 
\begin{align*}
& \left| 1+B_{k,n}(\l, x)\right|^2\\
&= \left(1-\frac{\l^2 \ckn^2}{2}  \cos^2 (m_k x) + \frac{\l^4 \ckn^4}{24} \cos^4(m_k x) \right)^2+\left(\l \ckn \cos(m_k x) - \frac{ \l^3 \ckn^3 }{6} \cos^3(m_k x) \right)^2 \\
&=  1-\frac{\l^6}{72} \ckn^6 \cos^6 (m_k x) + \frac{\l^8}{576}\ckn^8 \cos^8(m_k x) \leq 1+\frac{K^6}{72} \akn^6 + \frac{K^8}{576}\akn^8 .
\end{align*}
We conclude that term (\ref{rel:2}) is bounded by
\begin{align*}
&O\left(|\l|^5 \sum_{k=1}^n |\ckn|^5\right) \prod_{k=1}^n \sqrt{1+\frac{K^6}{72} \akn^6 + \frac{K^8}{576}\akn^8  }.
\end{align*}
Bearing in mind the notation (\ref{dn}), we derive 
\begin{align*}
&\int_{-KA_n}^{KA_n} \left|\frac{1}{2 \pi }\int_0^{2 \pi} \left(e^{    \gamma_n(\l,x) } -1\right) \prod_{k=1}^n \left( 1+B_{k,n}(\l, x)\right)   dx\right| d\l \\
& = O\left(A_n  K^5 \sum_{k=1}^n |\akn|^5  \prod_{k=1}^n  \sqrt{1+\frac{K^6}{72} \akn^6 + \frac{K^8}{576}\akn^8  }\right) \\
& = O\left( K^5 n^{3/2} d_n^6  e^{C n d_n^6}\right) \to 0,
\end{align*}
where we used conditions (\ref{dnq1}) and (\ref{dnq2}), to show the convergence to $0.$ Next we aim to show that the term (\ref{rel:1}) converges to $e^{-\frac{\l^2}{2}}$ in $L^1.$  We have
\begin{align}
&  1+B_{k,n}(\l, x) = D_{k,n}(\l)  \nonumber \\
&\cdot \left( 1+ \frac{F_{k,n}(\l)}{D_{k,n}(\l)} \cos( m_k x) + \frac{G_{k,n}(\l)}{D_{k,n}(\l)}  \cos(2 m_k x) + \frac{H_{k,n}(\l)}{D_{k,n}(\l)} \cos(3 m_k x) + \frac{J_{k,n}(\l)}{D_{k,n}(\l)}  \cos(4 m_k x)\right) ,\label{last_term}
\end{align}
where
\begin{align*}
&D_{k,n}(\l) = 1 - \frac{\l^2  \ckn^2}{4} + \frac{\l^4  \ckn^4}{64}, \quad F_{k,n}(\l) = i\l \ckn \left(1 - \frac{\l^2  \ckn^2}{8}  \right),\\
&G_{k,n}(\l) =  - \frac{\l^2  \ckn^2}{4} \left(1 + \frac{\l^2  \ckn^2}{6}\right), \quad H_{k,n}(\l) =- i \frac{\l^3  \ckn^3}{24} , \quad J_{k,n}(\l) = -\frac{\l^4  \ckn^4}{96}.
\end{align*}
We assume that $D_{k,n}(\l) \geq \frac{1}{2}$ for large enough $n$ (depending on $K$). From now on we treat the cases $1<q \leq 2$ and $q > 2,$ separately.
\begin{itemize}[leftmargin=0.15in]
\item  $1<q \leq 2:$ Using
\begin{align}
&\max_{1 \leq k \leq n, |\l| \leq KA_n} \lbrace \left|F_{k,n}(\l)\right| \rbrace\leq K d_n  \left(1+\frac{K^2 d_n^2}{6}\right), \label{bounds:1}\\
&\max_{1 \leq k \leq n, |\l| \leq KA_n} \lbrace \left|G_{k,n}(\l)\right|, \left|H_{k,n}(\l)\right|, \left|J_{k,n}(\l)\right| \rbrace\leq K^2 d_n^2  \left(1+\frac{K^2 d_n^2}{6}\right) \label{bounds:2},
\end{align}
we derive
\begin{align*}
&\left| e^{-\frac{\l^2}{2}} - \frac{1}{2 \pi }\int_0^{2 \pi} \prod_{k=1}^n \left( 1+B_{k,n}(\l, x)\right)  dx \right| \leq \left| e^{-\frac{\l^2}{2}} -\prod_{k=1}^n D_{k,n}(\l) \right|\\
&+\left(\prod_{k=1}^n D_{k,n}(\l) \right)\sum_{l=2}^n \sum_{p=0}^l  \left( K d_n  \left(1+\frac{K^2 d_n^2}{6}\right) \right)^l  \left( K d_n\right)^p C_r(l,p,q,n),
\end{align*}
where $C_r(l,p,q,n)$ is the number of solutions satisfying Equation (\ref{ptryg-equation}). Moreover, Lemma \ref{lemma:3.1} tells us that
$C_r(l,p,q,n) \leq \left(8 n \log_q(rl) \log_q \left(\frac{2r^2lq}{(q-1)^2}\right) \right)^{\frac{l+p}{3}},$ which implies that for every $\varepsilon > 0,$ there exist a constant $C_{\varepsilon, r}$ such that
\begin{align*}
C_r(l,p,q,n) \leq \left(  C_{\varepsilon, r} n^{1+\varepsilon}\right)^{\frac{l+p}{3}}.
\end{align*}
Thus,
\begin{align*}
&\sum_{l=2}^n \sum_{p=0}^l \ \left( K d_n  \left(1+\frac{K^2 d_n^2}{6}\right) \right)^l \left(Kd_n \right)^p C_r(l,p,q,n) \\
&\leq \left( \sum_{p=0}^\infty \left(  C^{1/3}_{\varepsilon, r} n^{\frac{1+\varepsilon}{3}} K d_n\right)^p \right)\sum_{l=2}^n   \left( K d_n \left(1+\frac{K^2 d_n^2}{6}\right) C^{1/3}_{\varepsilon, r} n^{\frac{1+\varepsilon}{3}}\right)^l \leq \left( d_n n^{\frac{1+\varepsilon}{3}}\right)^2 C_{n,K} ,
\end{align*}
where $C_{n,K}$ is independent of $\l$ and converges to some constant, when $n \to \infty$ (follows from the condition \ref{dnq1}).
On the other hand,
\begin{align}
& \left| e^{-\frac{\l^2}{2}} -\prod_{k=1}^n D_{k,n}(\l) \right| \leq e^{-\frac{\l^2}{2}} \left|1 - e^{\frac{\l^2}{2}+2\sum_{k=1}^n \log \left(1 - \frac{\l^2  \ckn^2}{8}  \right)} \right| \leq  C'  K^4 e^{-\frac{\l^2}{2}}  \sum_{k=1}^n \akn^4, \label{ebound}
\end{align}
for some constant $C'.$
Thus,
\begin{align*}
&\left| e^{-\frac{\l^2}{2}} - \frac{1}{2 \pi }\int_0^{2 \pi} \prod_{k=1}^n \left( 1+B_{k,n}(\l, x)\right)  dx \right| \leq C'  K^4 e^{-\frac{\l^2}{2}}  \sum_{k=1}^n \akn^4 +\left(\prod_{k=1}^n D_{k,n}(\l) \right) \left( d_n n^{\frac{1+\varepsilon}{3}}\right)^2 C_{n,K} \\
& \leq C'  K^4 e^{-\frac{\l^2}{2}}  \sum_{k=1}^n \akn^4 +\left| \prod_{k=1}^n D_{k,n}(\l) - e^{-\frac{\l^2}{2}} \right| \left( d_n n^{\frac{1+\varepsilon}{3}}\right)^2 C_{n,K} + e^{-\frac{\l^2}{2}}\left( d_n n^{\frac{1+\varepsilon}{3}}\right)^2 C_{n,K} \\
&\leq e^{-\frac{\l^2}{2}} \left( C'  K^4   \sum_{k=1}^n \akn^4 \left(1 + \left( d_n n^{\frac{1+\varepsilon}{3}}\right)^2 C_{n,K} \right) + \left( d_n n^{\frac{1+\varepsilon}{3}}\right)^2 C_{n,K}\right)\\
&\leq e^{-\frac{\l^2}{2}} \left( C'  K^4   n d_n^4 \left(1 + \left( d_n n^{\frac{1+\varepsilon}{3}}\right)^2 C_{n,K} \right) + \left( d_n n^{\frac{1+\varepsilon}{3}}\right)^2 C_{n,K}\right)
\end{align*}
which together with the condition (\ref{dnq1}) implies 
\begin{align*}
\int_{-KA_n}^{KA_n} \left| e^{-\frac{\l^2}{2}} - \frac{1}{2 \pi }\int_0^{2 \pi} \prod_{k=1}^n \left( 1+B_{k,n}(\l, x)\right)  dx \right| d\l \to 0,
\end{align*}
when $n \to \infty.$
\item  $q > 2:$ we keep using the bounds (\ref{bounds:1}) and (\ref{bounds:2}) for $F_{k,n}(\l)$ and $G_{k,n}(\l),$ respectively, and bound $H_{k,n}(\l)$ and $J_{k,n}(\l)$ as follows:
\begin{align*}
&\max_{1 \leq k \leq n, |\l| \leq KA_n} \lbrace \left|H_{k,n}(\l)\right|, \left|J_{k,n}(\l)\right| \rbrace \leq K^3 d_n^3  \left(1+\frac{K^2 d_n^2}{6}\right),
\end{align*}
implying
\begin{align*}
&\left| e^{-\frac{\l^2}{2}} - \frac{1}{2 \pi }\int_0^{2 \pi} \prod_{k=1}^n \left( 1+B_{k,n}(\l, x)\right)  dx \right| \leq \left| e^{-\frac{\l^2}{2}} -\prod_{k=1}^n D_{k,n}(\l) \right|\\
&+\left(\prod_{k=1}^n D_{k,n}(\l) \right)\sum_{l=2}^n \sum_{p_2=0}^l \sum_{p_3=0}^{l-p_2} \left( K d_n  \left(1+\frac{K^2 d_n^2}{6}\right) \right)^{l}  \left( K d_n\right)^{p_2+2p_3} C_r(l,p_2,p_3,q,n),
\end{align*}
where $C_r(l,p_2,p_3,q,n)$ is the number of solutions satisfying Equation (\ref{pqtryg-equation}). Moreover, Lemma \ref{lemma:3.2} tells us that
$C_r(l,p_2,p_3,q,n) \leq \left(20 n  \log_q\left( 2l r \right) \log_q\left( ql r \right) \log_q\left( \frac{4l^2 q^2 r^2}{q-2} \right)  \right) ^{\frac{l}{4}+\frac{p_2}{4}+\frac{p_3}{2}},$ which implies that for every $\varepsilon > 0,$ there exist a constant $C_{\varepsilon, r}$ such that
\begin{align*}
C_r(l,p,q,n) \leq \left(  C_{\varepsilon, r} n^{1+\varepsilon}\right)^{\frac{l}{4}+\frac{p_2}{4}+\frac{p_3}{2}}.
\end{align*}
Thus,
\begin{align*}
& \sum_{l=2}^n \sum_{p_2=0}^l \sum_{p_3=0}^{l-p_2} \left( K d_n  \left(1+\frac{K^2 d_n^2}{6}\right) \right)^{l}  \left( K d_n\right)^{p_2+2p_3} C_r(l,p_2,p_3,q,n) \\
&\leq \left( \sum_{p_2=0}^\infty \left(  C^{1/4}_{\varepsilon, r} n^{\frac{1+\varepsilon}{4}} K d_n\right)^{p_2} \right) \left( \sum_{p_3=0}^\infty \left(  C^{1/4}_{\varepsilon, r} n^{\frac{1+\varepsilon}{4}} K d_n\right)^{2p_3} \right) \\
&\cdot \sum_{l=2}^n   \left( K d_n  \left(1+\frac{K^2 d_n^2}{6}\right) C^{1/4}_{\varepsilon, r} n^{\frac{1+\varepsilon}{4}}\right)^l \leq \left( d_n n^{\frac{1+\varepsilon}{4}}\right)^2 C_{n,K} ,
\end{align*}
where $C_{n,K}$ is independent of $\l$ and converges to some constant, when $n \to \infty$ (follows from the condition (\ref{dnq2})), which together with (\ref{ebound}) implies
\begin{align*}
&\left| e^{-\frac{\l^2}{2}} - \frac{1}{2 \pi }\int_0^{2 \pi} \prod_{k=1}^n \left( 1+B_{k,n}(\l, x)\right)  dx \right| \leq C'  K^4 e^{-\frac{\l^2}{2}}  \sum_{k=1}^n \akn^4 +\left(\prod_{k=1}^n D_{k,n}(\l) \right) \left( d_n n^{\frac{1+\varepsilon}{4}}\right)^2 C_{n,K} \\
&\leq e^{-\frac{\l^2}{2}} \left( C'  K^4   n d_n^4 \left(1 + \left( d_n n^{\frac{1+\varepsilon}{4}}\right)^2 C_{n,K} \right) + \left( d_n n^{\frac{1+\varepsilon}{4}}\right)^2 C_{n,K}\right).
\end{align*}
As a result, we obtain
\begin{align*}
\int_{-KA_n}^{KA_n} \left| e^{-\frac{\l^2}{2}} - \frac{1}{2 \pi }\int_0^{2 \pi} \prod_{k=1}^n \left( 1+B_{k,n}(\l, x)\right)  dx \right| d\l \to 0,
\end{align*}
when $n \to \infty.$
\end{itemize}
The proof of the theorem follows. 
 \end{proof}

\section{Proof of Theorems \ref{walsh-modG}, \ref{walsh-LLT} and \ref{walsh-speed}} \label{Section:4}

Let $k = \sum_{i=0}^\infty x_i 2^i$ and $l = \sum_{i=0}^\infty y_i 2^i,$ where $x_i, y_i = 0$ or $1.$ Then we define 
\begin{align*}
k\oplus l = \sum _{i=0}^\infty |x_i - y_i| 2^i.
\end{align*}
Note that this operation is associative, i.e. $(k\oplus l) \oplus p = k\oplus (l \oplus p). $ Moreover, for any $k, l, p \in \mathbb{N},$ we have that $k \oplus l = p $ is equivalent to $k = l \oplus p.$ In addition, $W_k(x) W_l(x) = W_{k \oplus l} (x)$ for all $x.$ 
\begin{lemma} \label{lemma:4.1}
Let $(m_k)_{k\geq 1}$ be any sequence satisfying the condition (\ref{lacunary}) with $q\geq 2.$ There are no solutions satisfying 
\begin{align}\label{walsh-equation-0}
m_{k_1} \oplus m_{k_2} \oplus \cdots \oplus  m_{k_l} = 0, 
\end{align} 
where $1 \leq k_l < \dots < k_1 \leq n, l \in \mathbb{N}.$
 \end{lemma}
\begin{proof}
Let $\alpha$ be the largest exponent of $2$ with a nonzero coefficient in $m_{k_1}.$ As $\frac{m_{k_1}}{m_{k_2}} \geq 2,$  the largest exponent of $2$ with nonzero coefficient in $m_{k_2}$, can be at most $\alpha - 1.$ We obtain 
\begin{align*}
0 = m_{k_1} \oplus m_{k_2} \oplus \cdots \oplus  m_{k_l} \geq 2^\alpha > 0.
\end{align*}
Therefore, there are no $m_{k_1}, \dots, m_{k_l}$ satisfying Equation (\ref{walsh-equation-0}).
\end{proof}

\begin{lemma}\label{lemma:4.2}
Let $(m_k)_{k\geq 1}$ be a sequence satisfying the condition (\ref{lacunary}) with $1<q<2.$ The number of solutions $C(l,q,n)$ of the equation
\begin{align}\label{walsh-equation}
m_{k_1} \oplus m_{k_2} \oplus \cdots \oplus  m_{k_l} = A, 
\end{align} 
where $1 \leq k_l < \dots < k_1 \leq n, l \in \mathbb{N}, A\in \mathbb{Z}^{+},$ is at most $\left(2(\gamma + 7)n \log_q^2 (2) \right)^\frac{l}{3},$ where $\gamma$ is an integer such that $1 + \frac{1}{2^{\gamma}} \leq q < 1 + \frac{1}{2^{\gamma-1}}.$ 
\end{lemma}

\begin{proof}
Let $\alpha$ be the largest exponent of $2$ with a nonzero coefficient in $m_{k_1}.$ Then using Equation (\ref{walsh-equation}), we get
\begin{align}
A = m_{k_1} \oplus m_{k_2} \oplus \cdots \oplus  m_{k_l} \leq 2^\alpha + 2^{\alpha - 1} + \cdots + 1 < 2 m_{k_1}. \label{estim:4.2}
\end{align}
  We prove the theorem by using induction on $l.$ 
  
For $l=2,$ we have $m_{k_1} \oplus m_{k_2} = A.$ Let $\alpha$ be the largest exponent of $2$ with a nonzero coefficient in $m_{k_1}.$ Then we can write $m_{k_1} = 2^\alpha + d_{\a-1}2^{\alpha-1} + \cdots +d_0,$ where $d_i = 0$ or $1.$ We distinguish two possibilities.
\begin{itemize}[leftmargin=0.15in]
\item $2^\alpha \leq A:$ This together with the bound (\ref{estim:4.2}) implies that $2A > m_{k_1} >\frac{A}{2}.$ As a result the number of solutions is bounded by $\log_q(4).$ 
\item $2^\alpha > A:$ Then we have that the largest exponent of $2$ with a nonzero coefficient in $m_{k_2}$ must be again $\alpha.$  Suppose $m_{k_1}$ and $m_{k_2}$ are identical in all the upper $\alpha - \beta + 1$ entries, i.e. 
\begin{align*}
m_{k_1} &= 2^\alpha + d_{\alpha-1} 2^{\alpha-1} + \cdots +d_\beta 2^\beta + 2^{\beta - 1} + d_{\beta-2} 2^{\beta - 2} +\cdots + d_0,\\
m_{k_2} &= 2^\alpha + d_{\alpha-1} 2^{\alpha-1} + \cdots + d_\beta 2^\beta + d^\prime_{\beta-2} 2^{\beta - 2} + \cdots + d^\prime_0.
\end{align*}
 Thus we have
\begin{align*}
1+ \frac{1}{2^\gamma} \leq q \leq \frac{m_{k_1}}{m_{k_2}} \leq \frac{2^\alpha + d_{\alpha-1} 2^{\alpha-1} + \cdots + d_\beta 2^\beta + 2^{\beta - 1} + \cdots + 1}{2^\alpha + d_{\alpha-1} 2^{\alpha-1} + \cdots + d_\beta 2^\beta } \leq 1 + \frac{1}{2^{\alpha-\beta}}.
\end{align*} 
Hence, we have $\alpha - \beta \leq \gamma.$
On the other hand, $A=m_{k_1} \oplus m_{k_2}  \geq 2^{\beta-1} $ and we derive $\alpha \leq \log_2 2A + \gamma.$ Since $m_{k_1} < 2^{\alpha + 1},$ we obtain  $A < m_{k_1} < A 2^{\gamma + 2}$ (the left bound follows from (\ref{estim:4.2})). As a result, the number of solutions is bounded by $\log_q \left(2^{\gamma+2} \right).$  
\end{itemize}
We conclude that the number of solutions for $l=2$ is at most $(4+\gamma) \log_q (2).$

Next we assume that the statement of the lemma holds for all $l^\prime < l,$ we want to prove that it is true also for $l.$  We discuss two possibilities.
\begin{itemize}
\item $\frac{m_{k_1}}{m_{k_{2}}} \geq 2:$  Then $2^\alpha \leq A,$ which together with the estimate (\ref{estim:4.2}) implies that  $m_{k_1}$ can be chosen at most $\log_q (4)$ ways. We discuss two further  cases.
\begin{itemize}[leftmargin=0.15in]
\item $\frac{m_{k_2}}{m_{k_3}} < 2:$ Then we can choose $m_{k_2}$ at most in $n$ and $m_{k_3}$ in $\log_q (2)$ ways.
\item $\frac{m_{k_2}}{m_{k_3}} \geq 2:$ Let $\beta$ be the largest exponent of $2$ with a nonzero coefficient in $m_{k_2}.$ Then we have 
\begin{align*}
2^{\beta + 1} > m_{k_2} \oplus \cdots \oplus  m_{k_l} = A \oplus m_{k_1} \geq 2^{\beta} .
\end{align*}
 Thus, if $m_{k_1}$ is already chosen, we can choose $m_{k_2}$ and $m_{k_3}$ at most in $\log_q (4)$  and $n$ ways, respectively. 
\end{itemize}
We deduce that for $\frac{m_{k_1}}{m_{k_{2}}} \geq 2,$ the number of solutions is at most 
\begin{align}
n \log_q (4) \log_q (8) C(l-3,q,n). \label{bbb1}
\end{align}
\item $\frac{m_{k_1}}{m_{k_{2}}} < 2:$  We distinguish two further cases.
\begin{itemize}[leftmargin=0.15in]
\item $\frac{m_{k_2}}{m_{k_3}} < 2:$ We can choose $m_{k_1}$ at most in $n$ ways, then both $m_{k_2}$ and $m_{k_3}$ can be chosen at most  in $\log_q (2)$ ways.
\item $\frac{m_{k_2}}{m_{k_3}} \geq 2:$  As before let $\alpha$ be the largest exponent of $2$ with a nonzero coefficient in $m_{k_1}$. We discuss the following possibilities.
\begin{itemize}[leftmargin=0.15in]
 \item $2^{\alpha} \leq A:$ This together with the bound (\ref{estim:4.2}) implies that $2A > m_{k_1} >\frac{A}{2}.$ So we can choose $m_{k_1}, m_{k_2}$ and $m_{k_3}$ at most in $\log_q (4), \log_q (2)$ and $n$ ways, respectively. 
 \item $2^{\alpha} > A:$ Then the largest exponent of $m_{k_2}$ is $\a$ as well. Assuming $m_{k_1}$ and $m_{k_2}$ share the first $\alpha-\beta+1$ exponents in the dyadic expansion, we deduce $ \alpha - \beta \leq \gamma.$ 
 
If $2^{\beta-1} \leq A,$ then $A  < m_{k_1} < A 2^{\gamma + 2}$, thus $m_{k_1}$ can be chosen at most in $\log_q \left( 2^{\gamma+2} \right)$ ways, and $m_{k_2}$ and $m_{k_3}$ at most in $\log_q (2)$ and $n$ ways, respectively.

It remains to discuss the case $2^{\beta-1} > A$. There are at most $n$ choices for $m_{k_1}$ and $\log_q (2)$ choices for $m_{k_2}.$ As $2^{\beta-1} > A$, we deduce  $2^{\beta-1} \leq m_{k_3}.$ Hence, we obtain 
\begin{align*}
q^2 \leq \frac{m_{k_1}}{m_{k_3}} \leq 2^{\alpha - \beta + 2} \leq 2^{\gamma + 2}.
\end{align*}
 Thus, $m_{k_3}$ can be chosen at most in $\log_q \left( \frac{2^{\gamma + 2}}{q^2} \right) \leq \log_q \left( 2^{\gamma + 2} \right)$ ways. 
\end{itemize}
\end{itemize}
We conclude that the number of solutions for $\frac{m_{k_1}}{m_{k_2}} <2$ is at most 
\begin{align}
n \log_q (2) \log_q \left(2^{2\gamma+8} \right) C(l-3,q,n). \label{bbb2}
\end{align}
\end{itemize}
Summing up the bounds (\ref{bbb1}) and (\ref{bbb2}) and using the induction, we deduce that the number of solutions is at most
\begin{align*}
\left(n \log_q (2) \log_q \left(2^{2\gamma+14} \right) \right)^{\frac{l}{3}}
\end{align*}
and the proof follows.
 \end{proof}
\begin{proof}[Proof of Theorem \ref{walsh-modG}]
Let $\varphi_n(z)$ be the moment generating function of  $S^{\mathrm{W}}_n,$ i.e.
\begin{align*}
\varphi_n(z) = \E[e^{ z S^{\mathrm{W}}_n}] = \int_0^1 e^{ z S^{\mathrm{W}}_n(x) } dx = \int_0^1 e^{ z \sum_{k=1}^n \akn W_{m_k}(x) } dx.
\end{align*}
We first treat the case $q \geq 2.$  Using Lemma \ref{lemma:4.1}, we get
\begin{align}
\E[W_{m_{1}} \cdots W_{m_{n}}] = \E[W_{m_{1} \oplus  \dots \oplus m_{n}}] = 0 .\label{q2}
\end{align}
Next we aim to show that $W_{m_1}, \dots, W_{m_n} $ are indeed independent. Denoting by $Y_i := \frac{W_{m_i}+1}{2},$ we end up having Bernoulli$(\frac{1}{2})$ distributed $\left(Y_i\right)_{1 \leq i \leq n}$ random variables. Moreover, using the relation (\ref{q2}), we deduce that for every subset $1 \leq k_1 \leq \dots \leq k_p \leq n$ 
\begin{align}
\P[Y_{k_1} = 1, \dots, Y_{k_p} = 1] = \E[Y_{k_1} \cdots Y_{k_p}] = \E[Y_{k_1}] \cdots \E[Y_{k_p}] = \P[Y_{k_1} = 1] \cdots \P[ Y_{k_p} = 1]. \label{indep_var}
\end{align} 
This shows that $\left(Y_i\right)_{1 \leq i \leq n}$ are independent, which implies the independence of the random variables $\left(W_{m_i}\right)_{1 \leq i \leq n}.$ As a result, the moment generating function $\varphi_n(z)$ of $S^{\mathrm{W}}_n$ writes as 
 \begin{align*}
\E\left[e^{ z S^{\mathrm{W}}_n}\right]  &=  \prod_{k=1}^n \int_0^1 e^{ z  \akn W_{m_k}(x) } dx = \prod_{k=1}^n  \cosh \left(z\akn \right),
\end{align*}
so we have
\begin{align*}
\E\left[e^{ z S^{\mathrm{W}}_n}\right] e^{-\frac{A_n^2 z^2}{2}}  &=  e^{-\frac{A_n^2 z^2}{2}+ \sum_{k=1}^n \log  \left( \cosh \left(z\akn \right)\right)} .
\end{align*}
Let $|z| \leq n^{1/10}.$ Using $n d_n^5 \to 0,$ we derive $|z \akn | \leq 1$ for large enough $n,$ which together with the Taylor expansion implies
\begin{align*}
\E\left[e^{ z S^{\mathrm{W}}_n}\right] e^{-\frac{A_n^2 z^2}{2}} = e^{-\frac{z^4}{12} \sum_{k=1}^n \akn^4 + O\left( |z|^6 \sum_{k=1}^n \akn^6\right)} \to  e^{-\frac{z^4}{12} \kappa_4},
\end{align*}
as $n \to \infty.$
Now  we treat the case $1<q<2$ and let $|z| \leq n^{\min\lbrace \frac{\varepsilon}{3 } , \frac{1}{3}\rbrace}.$ We have
\begin{align*}
\E\left[e^{ z S^{\mathrm{W}}_n}\right] e^{-\frac{A_n^2 z^2}{2}} &= e^{-\frac{A_n^2 z^2}{2}}\int_0^1 e^{ z \sum_{k=1}^n \akn W_{m_k}(x) } dx \\
&= e^{-\frac{A_n^2 z^2}{2}} \int_0^1 \prod_{k=1}^n \left( \cosh \left(z\akn \right)+ W_{m_k}(x) \sinh \left(z\akn \right)\right)  dx\\
&= e^{-\frac{A_n^2 z^2}{2}} \left(\prod_{k=1}^n \cosh \left(z\akn \right)\right) \int_0^1 \prod_{k=1}^n \left( 1+ W_{m_k}(x) \tanh \left(z\akn \right)\right)  dx\\
&= e^{-\frac{A_n^2 z^2}{2}} \left(\prod_{k=1}^n \cosh \left(z\akn \right)\right) \\
&+ e^{-\frac{A_n^2 z^2}{2}} \left(\prod_{k=1}^n \cosh \left(z\akn \right)\right) \left(\int_0^1 \prod_{k=1}^n \left( 1+ W_{m_k}(x) \tanh \left(z\akn \right)\right)  dx - 1 \right).
\end{align*}
Using Lemma \ref{lemma:4.2}, we derive 
\begin{align*}
&\left|\int_0^1 \prod_{k=1}^n \left( 1+ W_{m_k} (x) \tanh \left(z\akn \right)\right)  dx - 1 \right| \leq \sum_{l=3}^n (|z| d_n)^l C(l,q,n) \\
&\leq \sum_{l=3}^n (|z| d_n)^l\left(2 (\gamma + 7)n  \log_q^2(2) \right)^{\frac{l}{3}} = O\left(\left(|z|d_n n^{1/3}\right)^3 \right) \to 0,
\end{align*}
where $C(l,q,n)$ is the number of solutions of Equation (\ref{walsh-equation}) and the convergence to $0$ follows from the condition $n^{1+\varepsilon} d_n^3 \to 0.$ We conclude
\begin{align*}
\E\left[e^{ z S^{\mathrm{W}}_n}\right] e^{-\frac{A_n^2 z^2}{2}} &= e^{-\frac{z^4}{12} \sum_{k=1}^n \akn^4 + O\left( |z|^6 \sum_{k=1}^n \akn^6\right)} \left(1 +  O\left(\left(|z|d_n n^{1/3}\right)^3 \right)\right) \to 1,
\end{align*}
as $n \to \infty$ and the proof of the theorem follows.
\end{proof}
\begin{proof}[Proof of Theorems \ref{walsh-LLT} and \ref{walsh-speed}]
In order to establish both theorems, it is enough to check that the mod-Gaussian convergence happens with a zone of control. Therefore we
check whether the condition \ref{Z1} is satisfied. We first discuss the case $q\geq 2$ and as before we assume $|\l| \leq n^{1/10}.$ Using inequality $|e^z - 1| \leq |z| e^{|z|}$ we derive 
\begin{align*}
&\left| \E \left[e^{ i \l S^{\mathrm{W}}_n} \right] e^{\frac{A_n^2 \l^2}{2}} -  1 \right| = \left|e^{\frac{A_n^2 \l^2}{2} + \sum_{k=1}^n \log\left(\cosh(\l \akn) \right)} -1\right|\\
& \leq \left|-\frac{\l^4}{12} \sum_{k=1}^n \akn^4 + O\left( \l^6 \sum_{k=1}^n \akn^6\right) \right|e^{\left|-\frac{\l^4}{12} \sum_{k=1}^n \akn^4 + O\left( \l^6 \sum_{k=1}^n \akn^6\right) \right| }.
\end{align*}
Since $\sum_{k=1}^n \akn^4 \to \kappa_4,$ there is a constant $C_1$ such that
$\left| \sum_{k=1}^n \akn^4 - \kappa_4 \right| \leq C_1$ for $n$ large enough. Moreover, $\l^6 \sum_{k=1}^n \akn^6 \leq \l^4 \left(\l^2 n d_n^6 \right) \leq C_2 \l^4.$ As a result
\begin{align*}
&\left| \E[e^{ i \l S^{\mathrm{W}}_n}] e^{\frac{A_n^2 \l^2}{2}} -  1 \right| \leq C_3 \l^4 e^{C_3 \l^4}, 
\end{align*}
where $C_3$ is a constant depending on $\kappa_4, C_1, C_2.$ We deduce that for $q\geq 2$ we have a zone of control with the parameters $\gamma = \frac{1}{10}$ and $v = w = 4.$ \\
Now we treat the case $1 < q < 2$ and let $|z| \leq n^{\min\lbrace \frac{\varepsilon}{3 } , \frac{1}{3}\rbrace}.$ Similar to the above case, we have
\begin{align*}
&\left| \E \left[e^{ i \l S^{\mathrm{W}}_n} \right] e^{\frac{A_n^2 \l^2}{2}} -  1 \right| \leq \left| e^{-\frac{A_n^2 z^2}{2}} \left(\prod_{k=1}^n \cosh \left(z\akn \right)\right) - 1 \right| \\
&+ \left| e^{-\frac{A_n^2 z^2}{2}} \left(\prod_{k=1}^n \cosh \left(z\akn \right)\right) \left(\int_0^1 \prod_{k=1}^n \left( 1+ W_{m_k} (x) \tanh \left(z\akn \right)\right)  dx - 1 \right) \right|\\
&\leq \left|-\frac{\l^4}{12} \sum_{k=1}^n \akn^4 + O\left( \l^6 \sum_{k=1}^n \akn^6\right) \right|e^{\left|-\frac{\l^4}{12} \sum_{k=1}^n \akn^4 + O\left( \l^6 \sum_{k=1}^n \akn^6\right) \right| }\\
&+ e^{\left|-\frac{\l^4}{12} \sum_{k=1}^n \akn^4 + O\left( \l^6 \sum_{k=1}^n \akn^6\right) \right| } O\left(\left(\l d_n n^{1/3}\right)^3 \right)
\leq C_3 \l^4 e^{C_3 \l^4 }.
\end{align*}
We conclude that for $1 < q < 2$ there is a zone of control with the parameters $\gamma = \min \lbrace \frac{\varepsilon}{3}, \frac{1}{3}\rbrace$ and $v=w=4.$
\end{proof}

\section{Proof of Theorems \ref{CLT_KAC_general}, \ref{CLT_shift}, \ref{mod-Gauss for Holder}, \ref{mod-Gauss specf}, \ref{spec-LLT} and \ref{spec-speed}} \label{Section:5}

\begin{proof}[Proof of Theorem \ref{CLT_KAC_general}] 
To prove the theorem, we first show that for $r \geq 1$ we have
\begin{align}
\frac{1}{n} \Vert \phi_r\left(x \right) + \cdots + \phi_r \left( 2^{n-1} x\right) \Vert_2^2 \leq \Vert \phi_r\Vert_2^2 + 2 \Vert\phi_r\Vert _2 \sum_{s \geq r+1} \Vert \phi_s \Vert_2. \label{lemma_for_CLT}
\end{align}
If we denote by $\Delta_n = f_n - f_{n-1}$ and $\Delta_1 = f_1,$ then $ \Vert \phi_n \Vert_2^2 = \sum_{k \geq n+1} \E\left[\Delta^2_k \right].$ Moreover, $f = \sum_{n \geq 1} \Delta_n$ in $L^2.$ We note
\begin{align*}
&\left| \left| \phi_r(x) + \cdots + \phi_r \left( 2^{n-1} x\right)\right| \right|_2^2 = \int \left( \phi_r(x) + \cdots + \phi_r\left( 2^{n-1} x \right)\right)^2 dx\\
&= \sum_{k=0}^{n-1} \int \phi_r^2\left( 2^k x\right) dx + 2 \sum_{0 \leq k < l \leq n-1} \int \phi_r\left( 2^k x\right) \phi_r\left( 2^l x\right) dx.
\end{align*} 
We first analyze the mixed terms.
\begin{align*}
&\int \phi_r\left( 2^k x\right) \phi_r\left( 2^l x\right) dx = \int \phi_r\left( x\right) \phi_r\left( 2^{l-k} x\right) dx \\
&= \int \left( \sum_{m \geq r+1} \Delta_m(x)\right) \left( \sum_{m' \geq r+1} \Delta_{m'}\left(2^{l-k}x \right)\right) dx = \sum_{m, m' \geq r+1} \int \Delta_m(x) \Delta_{m'} \left( 2^{l-k} x \right) dx.
\end{align*}
We note that if $m < m' + l-k,$ we get $\int \Delta_m(x) \Delta_{m'} \left( 2^{l-k} x \right) dx = 0.$ On the other hand, if $m > m' + l-k,$ we again obtain $\int \Delta_m(x) \Delta_{m'} \left( 2^{l-k} x \right) dx = 0,$ since $\E\left[\Delta_m | \mathcal{D}_{m'+l-k} \right] = 0.$ As a result, we deduce
\begin{align*}
&\int \phi_r\left( 2^k x\right) \phi_r\left( 2^l x\right) dx = \sum_{m \geq r+1} \int \Delta_{m+l-k}(x) \Delta_{m} \left( 2^{l-k} x \right) dx\\
&\leq \int \sqrt{\sum_{m \geq r+1} \Delta_{m+l-k}^2 (x)} \sqrt{\sum_{m \geq r+1} \Delta_m^2\left(2^{l-k}x \right)}dx\\
&\leq \left( \int \sum_{m \geq r+1} \Delta_{m+l-k}^2\left(x\right) dx \right)^{1/2} \left( \int \sum_{m \geq r+1} \Delta_{m}^2 \left(2^{l-k}x \right) dx \right)^{1/2} \leq \left| \left| \phi_r \right| \right|_2 \, \left| \left|\phi_{r+l-k} \right| \right|_2.
\end{align*}
We conclude
\begin{align*}
&\left| \left| \phi_r(x) + \cdots + \phi_r \left( 2^{n-1} x\right)\right| \right|_2^2 \leq n || \phi_r||_2^2 + 2 \sum_{0 \leq k < l \leq n-1} ||\phi_r||_2 \, ||\phi_{r+l-k}||_2 \\
& \leq n \Vert \phi_r\Vert_2^2 + 2  \Vert\phi_r\Vert_2 \left(  \Vert \phi_{r+1}\Vert_2 \left(n-1 \right) + \cdots +  \Vert\phi_{r+n-1}\Vert_2 \right)\\
& \leq n \Vert \phi_r\Vert_2^2 + 2  \Vert\phi_r\Vert_2 n \left(  \Vert\phi_{r+1}\Vert_2  + \cdots +  \Vert\phi_{r+n-1} \Vert_2 \right).
\end{align*}
As a result, the inequality (\ref{lemma_for_CLT}) holds. To show the central limit theorem, we write 
\begin{align*}
\frac{1}{\sqrt{n}} \left( f(x) + \cdots +f\left(2^{n-1} x \right) \right) &= \frac{1}{\sqrt{n}} \left(f_r(x) + \cdots + f_r \left(2^{n-1}x \right) \right) \\
&+ \frac{1}{\sqrt{n}} \left(\phi_r(x) + \cdots + \phi_r \left(2^{n-1} x\right) \right).
\end{align*}
The first term converges in law to a normal variable with mean zero and variance 
\begin{align*}
\frac{1}{n} \int \left(f_r(x) + \cdots + f_r\left(2^{n-1} x \right) \right)^2 dx.
\end{align*}
Indeed for all $k,$  $f_r(x) + \cdots + f_r \left( 2^{k-1} x \right)$ and $\left( f_r \left( 2^{k+r} x \right) + \cdots + f_r \left( 2^{k+r+s} x \right) \right)$ are independent. So the central limit theorem follows from the central limit theorem for $m$-dependent variables \cite{Diananda}.
The second term is arbitrary small for $r$ large. Therefore, 
\begin{align*}
\frac{f(x) + \dots + f\left( 2^{n-1} x\right)}{\sqrt{n}} \to N(0, \sigma^2),
\end{align*}
with $\sigma^2 = \lim_{n \to \infty} \frac{1}{n} \int\left( f(x) + \cdots + f\left(2^{n-1}x\right)\right)^2 dt.$ Next we show that the limit for $\sigma^2$ exists. The limit for $f_r$ clearly exists and the term for $\phi_r$ is arbitrary small. Therefore, for each $\varepsilon > 0$ there exists $r$ such that
\begin{align*}
& \lim_{n \to \infty} \int dx \frac{\left(f_r(x) + \cdots + f_r\left(2^{n-1} x \right) \right)^2}{n} - \varepsilon \leq \underline{\lim} \int dx \frac{\left(f(x) + \cdots + f\left(2^{n-1} x\right) \right)^2}{n}\\
&\leq \overline{\lim} \int dx \frac{\left(f(x) + \cdots + f\left(2^{n-1} x\right) \right)^2}{n} \leq \lim_{n \to \infty} \int dx \frac{\left(f_r(x) + \cdots + f_r\left(2^{n-1} x \right) \right)^2}{n} + \varepsilon.
\end{align*}
And the proof of the theorem follows.
\end{proof}
\begin{proof}[Proof of Theorem \ref{CLT_shift}]
As before, we denote by $\Delta_k = f_k - f_{k-1} $ so $f=\sum_k \Delta_k.$
We first show 
\begin{align}
\frac{1}{n} \E \left[\left(\phi_r + \dots + \phi_r \circ \theta^{n-1} \right)^2 \right] \leq ||\phi_r||_2^2 + 2 ||\phi_r||_2 \sum_{s \geq r+1} ||\phi_s||_2. \label{lemma_shift}
\end{align}
Note 
\begin{align*}
&\E \left[\left(\phi + \cdots + \phi_r \circ \theta^{n-1} \right)^2 \right] = \sum_{k=0}^{n-1} ||\phi_r \circ \theta^k||_2^2 + 2 \sum_{0 \leq k , l \leq n-1} \E \left[\phi_r \circ \theta^k \phi_r \circ \theta^l \right]\\
&=n || \phi_r||_2^2 + 2 \sum_{0 \leq k < l \leq n-1} \E \left[\phi_r \circ \theta^k \phi_r \circ \theta^l \right].
\end{align*}
Using $\phi_r = \sum_{m \geq r+1} \Delta_m,$ we get that the second term for the fixed $k,l,$ 
\begin{align*}
\E \left[\phi_r \circ \theta^k \phi_r \circ \theta^l \right] = \sum_{m,m' \geq r+1} \E[\Delta_m \circ \theta^k \Delta_{m'} \circ \theta^l] = \E \left[ \Delta_m \Delta_{m'} \circ \theta^{l-k}\right].
\end{align*}
For $m > m'+l-k,$ $\Delta_{m'} \circ \theta^{l-k}$ is $\mathcal{F}_{m' + l - k}$ measurable and $\E\left[\Delta_m | \mathcal{F}_{m' + l - k} \right] = 0.$ So the term is equal to $0.$ For $m < m'+l-k,$ we find 
\begin{align*}
&\E \left[\Delta_m \Delta_{m'} \circ \theta^{l-k} \right] = \int \mu \left(d y_1 \right) \cdots \mu \left(d y_{m'+l-k} \right) \Delta_m \left(y_1, \dots, y_m \right) \Delta_{m'} \left(y_{1+l-k}, \dots, y_{m'+l-k} \right)\\
&= \int \mu \left(d y_1 \right) \cdots \mu \left(d y_{m} \right)\int \mu \left(d y_{m+1} \right) \cdots \mu \left(d y_{m'+l-k} \right)  \Delta_m \left(y_1, \dots, y_m \right) \Delta_{m'} \left(y_{1+l-k}, \dots, y_{m'+l-k} \right)\\
&= \int \mu \left(d y_1 \right) \cdots \mu \left(d y_{m} \right) \Delta_m \left(y_1, \dots, y_m \right) \int \mu \left(d y_{m+1} \right) \cdots \mu \left(d y_{m'+l-k} \right)   \Delta_{m'} \left(y_{1+l-k}, \dots, y_{m'+l-k} \right).
\end{align*}
The last integral is calculated as follows
\begin{align*}
&\int \mu \left(d y_{m+1} \right) \cdots \mu \left(d y_{m'+l-k} \right)   \Delta_{m'} \left(y_{1+l-k}, \dots, y_{m'+l-k} \right) = \E \left[\Delta_{m'} | \mathcal{F}_{m'+l-k-m} \right] \left(y_{1+l-k}, \dots, y_{m'+l-k-m} \right).
\end{align*}
If $m'+(l-k) - m < m',$ then $l-k <m,$ we get $0.$
If $l-k \geq m,$ then $\Delta_m$ and $\Delta_{m'} \circ \theta^{l-k}$ are independent, thus we get $0.$
So we end up having only terms with $m' + l - k  = m.$
As a result,
\begin{align*}
& \sum_{0 \leq k < l \leq n-1} \sum_{m, m' \geq r+1} \E \left[\Delta_m \Delta_{m'} \circ \theta^{l-k} \right] = \sum_{0 \leq k < l \leq n-1}\sum_{m' \geq r+1,  m = m'+l-k} \E \left[ \Delta_m \Delta_{m'} \circ \theta^{l-k}\right]\\
&=\sum_{0 \leq k < l \leq n-1}\sum_{m' \geq r+1} \E \left[ \Delta_{m'+l-k} \Delta_{m'} \circ \theta^{l-k}  \right]\\
&\leq \sum_{0 \leq k < l \leq n-1} \E\left[\sqrt{\sum_{m' \geq r+1}  \Delta^2_{m'+l-k}} \sqrt{\sum_{m' \geq r+1}  \Delta^2_{m'} \circ \theta^{l-k}} \right]\\
&\leq \sum_{0 \leq k < l \leq n-1} \Vert \phi_{r+l-k} \Vert_2 \Vert \phi_{r} \Vert_2 \leq \Vert \phi_{r} \Vert_2 \left((n-1) \Vert \phi_{r+1} \Vert_2 + \cdots + \Vert \phi_{r+n-1} \Vert _2 \right)\\
&\leq n \Vert \phi_{r} \Vert_2 \sum_{s \geq r+1} \Vert \phi_{s} \Vert_2,
\end{align*}
which proves the inequality (\ref{lemma_shift}).
Finally, as for all $r$ we have $f_r$ and $f_r \circ \theta^r$ are independent, we can apply the central limit theorem for $m$-dependent random variables presented in \cite{Diananda} to 
\begin{align*}
\frac{f_r + f_r \circ \theta + \cdot + f_r \circ \theta^{n-1}}{\sqrt{n}},
\end{align*}
which together with the inequality (\ref{lemma_shift}) proves the theorem. The details are similar to the proof of Theorem \ref{CLT_KAC_general}.
\end{proof}

\begin{proof}[Proof of Theorem \ref{mod-Gauss for Holder}]
We divide the interval $[0,1]$  into $b_k$ parts and choose arbitrary numbers $x_{j,k}$ from these intervals. We define $g_k(x) = f(x_{j,k})$ for $\frac{j}{b_k} \leq x < \frac{j+1}{b_k},  \quad j=0,1,\dots,b_k-1$ and extend $g_k(x)$ to $\mathbb{R}$ by taking $g_k(x+1) = g_k(x).$ We obviously have $\left| g_k(x) - f(x) \right| < \frac{h}{b_k^{\alpha}}, \left|\int_0^1 g_k(x) dx\right| <  \frac{h}{b_k^{\alpha}}$ and  $\left|\int_0^1 g_k^2(x)dx - 1\right| <  \frac{2hM}{b_k^{\alpha}},$ where $M = \sup_{x \in[0;1]} |f(x)|.$\\
 Next we define $f_k(x) := g_k(m_k x).$ Note that since $\frac{m_{k+1}}{m_k} = b_k \in \lbrace 2, 3, \dots \rbrace$ the functions $\left(f_k(x) \right)_{k \geq 1}$  are independent. If we denote by $\mu_k := \int_0^1 f_k(x) dx$ and $\sigma_k^2 := \int_0^1 f_k^2(x)dx - \left( \int_0^1 f_k(x) dx\right)^2,$ then
 \begin{align*}
 &\left| f_k(x) -  f \left( m_k x \right) \right| < \frac{h}{b_k^{ \alpha}},\\
 &\left|\mu_k\right| <  \frac{h}{b_k^{\alpha}},\\
 &\left|\sigma_k^2 - 1\right| <  \frac{2hM}{b_k^{\alpha}} + \frac{h^2}{b_k^{2 \alpha}}.
\end{align*}  
Let $\varphi_n(i \l)$ be a characteristic function of $S^{\mathrm{H}}_n$, i.e.
\begin{align*}
\varphi_n(i \lambda) = \int_0^1 e^{ i \lambda S^{\mathrm{H}}_n(x) } dx.
\end{align*} 
Note that in order to show the mod-Gaussian convergence in the sense of Definition \ref{mod-Gauss-weak}, it suffices to show that 
\begin{align*}
\varphi_n \left( \frac{i \l}{A_n}\right) \mathbbm{1}_{ \lbrace |\l| \leq A_n K\rbrace }  \to e^{- \frac{\l^2}{2}}.
\end{align*}
in $L^1$ for all $K \geq 0.$
Using the notation (\ref{ckn}), we have 
\begin{align} 
&\varphi_n \left( \frac{i \l}{A_n}\right) = \int_0^{1} e^{ i \l \sum_{k=1}^n \ckn f\left(m_kx\right) } dx  \nonumber \\
& =  \int_0^{1} \left( e^{ i \l \sum_{k=1}^n \ckn f\left(m_kx\right) }  -  e^{ i \l \sum_{k=1}^n \ckn f_k(x) }\right) dx \label{3.1.1} \\
& + \int_0^{1}  e^{ i \l \sum_{k=1}^n \ckn f_k(x) } dx \label{3.1.2}
\end{align}
To estimate (\ref{3.1.1}), note that 
\begin{align*}
 & \left|\int_0^{1} \left( e^{ i \l \sum_{k=1}^n \ckn f\left(m_kx\right) }  -  e^{ i \l \sum_{k=1}^n \ckn f_k(x) }\right) dx\right| \\
 & \leq \int_0^{1} \left| e^{ i \l \sum_{k=1}^n \ckn \left(f\left(m_kx\right) - f_k(x) \right)}  - 1\right| dx = O\left( \l \sum_{k=1}^n |\ckn| \frac{h}{b_k^{ \alpha}} \right) .
\end{align*}
Thus, $\int_{-KA_n}^{KA_n} \left|\int_0^{1} \left( e^{ i \l \sum_{k=1}^n \ckn f\left(m_kx\right) }  -  e^{ i \l \sum_{k=1}^n \ckn f_k(x) }\right) dx\right| = O\left(A_n \sum_{k=1}^n|\akn| \frac{h}{b_k^\alpha}  \right) \to 0,$ when $n\to \infty.$ Next using the independence, (\ref{3.1.2}) can be written as 
\begin{align*} 
&\int_0^{1} e^{ i \l \sum_{k=1}^n \ckn f_k(x) } dx  = \prod_{k=1}^n \int_0^{1} e^{ i \l \ckn f_k(x) } dx = \prod_{k=1}^n e^{  i \l \ckn \mu_k - \half \l^2 \ckn^2 \sigma^2_k    + \gamma_{k,n}(\l)}\\
&= e^{-\frac{ \l^2}{2}} + e^{-\frac{\l^2}{2}} \left( e^{  i \l \sum_{k=1}^n \ckn \mu_k - \half \l^2 \sum_{k=1}^n \ckn^2 \left( \sigma^2_k - 1\right) + \sum_{k=1}^n \gamma_{k,n}(\l)} -1\right),
\end{align*}
where $|\gamma_{k,n}(\l)| = O(|\l \ckn|^3).$
So it remains to show that the last summand converges to $0$ in $L^1$ when $n\to \infty.$ Note that 
\begin{align*}
&\int_{-KA_n}^{KA_n} e^{-\frac{\l^2}{2}} \left| e^{  i \l \sum_{k=1}^n \ckn \mu_k - \half \l^2 \sum_{k=1}^n \ckn^2 \left( \sigma^2_k - 1\right) + \sum_{k=1}^n \gamma_{k,n}(\l)} -1\right| d\l \\
&= \int_{-KA_n}^{KA_n} e^{-\frac{\l^2}{2}} O \left( |\l| \sum_{k=1}^n |\ckn \mu_k| + \half \l^2 \sum_{k=1}^n \ckn^2 \left| \sigma^2_k - 1\right| + \sum_{k=1}^n |\gamma_{k,n}(\l)| \right) d\l \\
&= O \left( K \sum_{k=1}^n |\akn \mu_k| + \half K^2 \sum_{k=1}^n \akn^2 \left| \sigma^2_k - 1\right| + K^3 \sum_{k=1}^n |\akn|^3 \right)\int_{-KA_n}^{KA_n} e^{-\frac{\l^2}{2}}  d\l, \\
\end{align*}
which goes to $0$, as $n \to \infty.$ Hence, the proof of the theorem follows.
\end{proof}

 \begin{proof}[Proof of Theorem \ref{mod-Gauss specf}]
 We write 
 \begin{align*}
 x - [x]  - \half = - \frac{r_1(x)}{2^2} - \frac{r_2(x)}{2^3} - \cdots, 
\end{align*}   
where $\left( r_n(x)\right)_{n \geq 0}$ are the Rademacher functions. Note that $r_n(2^k x) = r_{n+k} (x),$ hence, 
$f \left(2^kx\right) = - \sum_{l=1}^\infty \frac{r_{l+k} (x)}{2^{l+1}}.$
As a result,
\begin{align*}
&\sum_{k=1}^n f \left(2^kx\right) = - \sum_{k=1}^n \sum_{l=1}^\infty \frac{r_{l+k} (x)}{2^{l+1}} = - \sum_{k=1}^n \sum_{p=k+1}^\infty \frac{r_{p} (x)}{2^{p-k+1}}=- \sum_{p=2}^{n+1} \sum_{k=1}^{p-1} \frac{r_{p} (x)}{2^{p-k+1}} - \sum_{p=n+2}^\infty \sum_{k=1}^n \frac{r_{p} (x)}{2^{p-k+1}} \\
&= - \sum_{p=2}^{n+1} r_{p} (x) \left(\frac{1}{2}-\frac{1}{2^p}\right)  - \sum_{p=n+2}^\infty r_{p} (x) \left( \frac{1}{2^{p-n}} - \frac{1}{2^p}\right).
\end{align*}
We assume that $|z| \leq n^{1/24},$ so the moment generating function of $S_n$ can be modified in the following way:
 \begin{align*}
 &E\left[e^{z S_n}\right]  = E\left[e^{z \sum_{k=1}^n f(2^kx)}\right] \\
 &=\prod_{p=2}^{n+1} E\left[e^{-z r_p(x) \left(\frac{1}{2}-\frac{1}{2^p}\right) }\right] \prod_{p=n+2}^{\infty} E\left[e^{-z r_p(x) \left(\frac{1}{2^{p-n}}-\frac{1}{2^p}\right) }\right]   \\
 &= e^{\sum_{p=2}^{n+1} \log \left(\cosh \left(z  \left(\frac{1}{2}-\frac{1}{2^p}\right) \right) \right)+\sum_{p=n+2}^{\infty}\log\left(\cosh \left( -z  \left(\frac{1}{2^{p-n}}-\frac{1}{2^p}\right) \right) \right) }   \\
 &= e^{ \frac{z^2}{2} \left( \sum_{p=0}^{n-1}    \left(\frac{1}{2}-\frac{1}{2^{p+2}}\right)^2+\sum_{p=0}^\infty    \left(\frac{1}{2^{p+2}}-\frac{1}{2^{p+n+2}}\right)^2 \right) -\frac{z^4}{12}\left( \sum_{p=0}^{n-1}    \left(\frac{1}{2}-\frac{1}{2^{p+2}}\right)^4+\sum_{p=0}^\infty    \left(\frac{1}{2^{p+2}}-\frac{1}{2^{p+n+2}}\right)^4 \right)+  O\left( |z|^6n\right) }   \\
&= e^{ \frac{z^2}{2} \left(- \frac{5}{12}+  \frac{1}{2^{ n+1}} - \frac{1}{3 \cdot 2^{2n+2}} + \frac{n}{4}+\frac{1}{12} \left(1 - \frac{1}{2^n} \right)^2 \right) } \\
&\cdot e^{-\frac{z^4}{12 } \left( -\frac{1}{15 \cdot 2^{4+4n}} - \frac{263}{ 105 \cdot  2^{4 }} + \frac{1}{7 \cdot 2^{2 +3 n}} - \frac{1}{ 2^{3 + 2 n}} + \frac{1}{ 2^{2 +  n}} + \frac{ n}{16} + \frac{1}{15 \cdot 2^4}  \left(1 -\frac{1}{2^n} \right)^4\right)  +  O\left( |z|^6n\right)} \\
&= e^{ \frac{z^2}{2} \left(- \frac{1}{3}+  \frac{1}{3 \cdot 2^{ n}}  + \frac{n}{4} \right)-\frac{z^4}{12 } \left(  - \frac{16}{ 105 } + \frac{1}{105 \cdot 2^{3 n-1}} - \frac{1}{ 5 \cdot 2^{1 + 2 n}} + \frac{7}{15 \cdot  2^{1 +  n}} + \frac{ n}{16} \right)  +  O\left( |z|^6n\right)}. 
\end{align*}  
Thus, we obtain 
\begin{align*}
&\E\left[e^{\frac{z S_n}{n^{1/4}}}\right] e^{- \frac{z^2\sqrt{n}}{8}} =   
e^{ \frac{z^2}{2\sqrt{n}} \left(- \frac{1}{3}+  \frac{1}{3 \cdot 2^{ n}}  \right)-\frac{z^4}{12n } \left(  - \frac{16}{ 105 } + \frac{1}{105 \cdot 2^{3 n-1}} - \frac{1}{ 5 \cdot 2^{1 + 2 n}} + \frac{7}{15 \cdot  2^{1 +  n}} + \frac{ n}{16} \right)  +  O\left(\frac{ |z|^6}{\sqrt{n}}\right)} \to e^{-\frac{z^4}{192}}.
\end{align*}
Hence, the proof of the theorem follows.
 \end{proof}
 
\begin{proof}[Proof of Theorems \ref{spec-LLT} and \ref{spec-speed}]
In order to establish both theorems, it is enough to check that the mod-Gaussian convergence happens with a zone of control. Therefore we
check whether the conditions \ref{Z1} and \ref{Z2} are satisfied. Using the proof of Theorem \ref{mod-Gauss specf} and the approach used to prove Theorems \ref{walsh-LLT} and \ref{walsh-speed}, we obtain that there is a zone of control with the parameters $\gamma = \frac{1}{24}, \quad v = 2, \quad w = 4.$ We leave the details to the reader.
\end{proof}
\section{Acknowledgments}
We would like to thank Prof. Ashkan Nikeghbali (Universit\"at Z\"urich) for helpful discussions.

\end{document}